\documentclass{article}

\usepackage{arxiv}

\usepackage[utf8]{inputenc} % allow utf-8 input
\usepackage[T1]{fontenc}    % use 8-bit T1 fonts
\usepackage[hidelinks]{hyperref}       % hyperlinks
\usepackage{url}            % simple URL typesetting
\usepackage{booktabs}       % professional-quality tables
\usepackage{amsfonts}       % blackboard math symbols
\usepackage{amsmath}
\usepackage{amssymb}
\usepackage{amsthm}
\usepackage{nicefrac}       % compact symbols for 1/2, etc.
\usepackage{microtype}      % microtypographz
\usepackage{graphicx}
\usepackage[square,sort,comma,numbers]{natbib}
\usepackage{doi}
\usepackage{float}
\usepackage{subfigure,color}
\usepackage{lipsum}

\newtheorem{theorem}{Theorem}[section]
\newtheorem{remark}[theorem]{Remark}

\newtheorem{proposition}[theorem]{Proposition}

\newcounter{lettercounter}

\title{A spectral element solution of the Poisson equation with shifted boundary polynomial corrections: influence of the surrogate to true boundary mapping and an asymptotically preserving Robin formulation}

\author{ \href{https://orcid.org/0000-0001-6698-2623}{\includegraphics[scale=0.06]{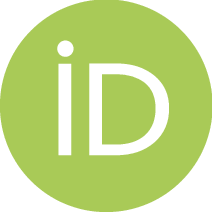}\hspace{1mm}Jens Visbech} \\
	Department of Applied Mathematics and Computer Science\\
	Technical University of Denmark\\
	Kongens Lyngby, 2800, Denmark \\
	\texttt{jvis@dtu.dk} \\
\And
	\href{https://orcid.org/0000-0001-8626-1575}{\includegraphics[scale=0.06]{Figures/orcid.eps}\hspace{1mm}Allan Peter Engsig-Karup} \\
	Department of Applied Mathematics and Computer Science\\
	Technical University of Denmark\\
	Kongens Lyngby, 2800, Denmark \\
	\texttt{apek@dtu.dk} \\
\And
	\href{https://orcid.org/0000-0002-1679-7339}{\includegraphics[scale=0.06]{Figures/orcid.eps}\hspace{1mm}Mario Ricchiuto} \\
    INRIA, Univ. Bordeaux, CNRS, Bordeaux INP, IMB, UMR 5251, \\200 Avenue de la Vieille Tour, 33405 Talence cedex, France\\
	\texttt{mario.ricchiuto@inria.fr} \\
}

\renewcommand{\headeright}{Visbech et al. (2023)}
\renewcommand{\shorttitle}{An \textit{arXiv} preprint}

\hypersetup{
pdftitle={},
pdfauthor={Jens Visbech, Allan Peter Engsig-Karup, Mario Ricchiuto}
}

\begin{document}

% \input{ToDo.tex}

% \newpage $~~~~$
% \newpage

\maketitle

\begin{abstract}
	We present a new high-order accurate spectral element solution to the two-dimensional scalar Poisson equation subject to a general Robin boundary condition. The solution is based on a simplified version of the shifted boundary method employing a continuous arbitrary order $hp$-Galerkin spectral element method as the numerical discretization procedure. The simplification relies on a polynomial correction to avoid explicitly evaluating high-order partial derivatives from the Taylor series expansion, which traditionally have been used within the shifted boundary method. In this setting, we apply an extrapolation and novel interpolation approach to project the basis functions from the \textit{true} domain onto the approximate \textit{surrogate} domain. The resulting solution provides a method that naturally incorporates curved geometrical features of the domain, overcomes complex and cumbersome mesh generation, and avoids problems with small-cut-cells. Dirichlet, Neumann, and general Robin boundary conditions are enforced weakly through: i) a generalized Nitsche's method and ii) a generalized Aubin's method. For this, a consistent asymptotic preserving formulation of the embedded Robin formulations is presented.

We present several numerical experiments and analysis of the algorithmic properties of the different weak formulations. With this, we include convergence studies under polynomial, $p$, increase of the basis functions, mesh, $h$, refinement, and matrix conditioning to highlight the spectral and algebraic convergence features, respectively. This is done to assess the influence of errors across variational formulations, polynomial order, mesh size, and mappings between the true and surrogate boundaries.
\end{abstract}


 \section{Introduction}\label{sec:introduction} % Intro - To the point:

The need for accurate representation of solutions to partial differential equations (PDEs) -- that take complex geometry into account -- appears frequently in real-world science and engineering applications. Some applications largely involve smooth fields, for example, free surface dispersive wave propagation, electro-magnetics, and electrostatics, among others. In these cases, very precise predictions can be obtained using, e.g., high-order finite elements on conformal meshes. However, the computational time spent meshing is often relatively high. The presence of moving boundaries, e.g., as in free surface potential flow wave models, further complicates this challenging task. Unfitted techniques -- such as immersed and embedded methods -- allow to work with meshes agnostic of the domain boundaries, hereby lighting the meshing task and offering greater geometrical flexibility. Another important aspect is the type of boundary condition being imposed. Many applications, such as some of those mentioned above, involve a mix of Dirichlet, Neumann, as well as Robin conditions. The latter plays an important role in many problems. To mention a few:  relatively simple elliptic and parabolic PDEs, as heat transfer and fluid flow \cite{Li08,LiWang20,ACHDOU1998187}, acoustic wave propagation in different media \cite{AntoineBarucq05,DiazPeron16,PIND2021107596}, electromagnetic wave scattering as well as electrostatics \cite{AH98,BB02,Haddar05,Aslan11,PERRUSSEL201348}. An interesting aspect is that, in many cases, Robin conditions allow us to model the presence of an unresolved thin boundary layer and present themselves as a small perturbation of a Dirichlet (or Neumann) condition. This hides the asymptotic behavior of the medium being modeled within such layer, e.g., see \cite{AV18} and references therein for a review. In more complex settings, e.g., as fluid-structure interaction, the Robin condition is also known to play a major role \cite{FSI1,FSI2}. An additional very appealing use of Robin conditions is their potential for managing the variability of the boundary conditions in space and time -- with a single implementation -- as shown in \cite{refId0}.

In this paper, we are interested in studying high-order unfitted methods, and in particular Robin formulations, that can provide high-order accuracy on unfitted meshes. This can be used to switch from one boundary condition to another with a single implementation, which is consistent with asymptotic approximations possibly embedded in the PDE. The focus of the paper is on the Poisson equation with a reaction term as this is the governing spatial operator in many of the physical applications mentioned above.

% Why use high-order?
\underline{\textit{High-order $hp$-spectral element methods.}} Many problems are well known to be more efficiently approximated by high-order methods. In this paper, we target a very high-order method with convergence rates much higher than second-order. The known gains obtainable with such methods are associated with faster convergence if the solution has sufficient smoothness or appropriate $hp$-adaptive strategies are adopted \cite{Kreiss1972,hesthaven_gottlieb_gottlieb_2007,Xu2018}.  

% SEM - A high-order method:
Among high-order numerical discretizations, the spectral element method (SEM) \cite{Xu2018}, or $hp$-FEM \cite{Bakuska1990}, are classical choices. The SEM is a combination of two classes of methods: a high-order accurate spectral collocation method and a finite element method that is designed to handle geometry in a flexible way. It was introduced by Patera in the 1980s for solving laminar flows \cite{Patera1984}. If the solution to the mathematical problem stated in terms of PDEs is sufficiently smooth, the errors will decay at an exponential/spectral rate when increasing the approximation degree of the polynomial basis, $p$. Also, compared to spectral collocation methods, the resulting linear system of equations has a high degree of sparsity due to the domain decomposition of the domain into elements. Ultimately, this provides efficient linear scaling opportunities when dealing with complex geometry, e.g., through geometric $p$-multigrid methods \cite{ronquist1987,EngsigKarup2021}. The SEM also allows for the use of simplicial meshes, which have the benefit of great geometrical flexibility where necessary. In this work, SEM approximations up to an order of $11$ will be considered since this is assumed to cover most practical needs.

% Introducing the unfitted boundary methods
\underline{\textit{High order unfitted approximation and shifted boundary approach.}} As already mentioned, unfitted methods avoid cumbersome mesh generation and mesh management issues in applications involving complex geometries or moving boundaries.
There exist many different approaches to constructing such methods. A widely used approach is based on volume penalization, or forcing, which originates from the work of Peskin in the 1970s \cite{Peskin1972}. The method is quite robust and has been applied in many applications (see the recent review \cite{verzicco23}) and to impose several types of boundary conditions, including Neumann and Robin \cite{THIRUMALAISAMY2022110726,RAMIERE2007766,BENSIALI2015134}. However, its accuracy is often limited to first or second-order, which requires advanced use of adaptive techniques to offset the associated error  \cite{Mittal2005,verzicco23}. The finite element setting has shown a greater potential to allow for the construction of higher accuracy unfitted methods. Starting with the early works on GFEM, XFEM, and cut finite elements \cite{Strouboulis2000,Moes1999,Hansbo2002}, many developments have allowed us to tackle some of the underlying initial limitations of methods based on mesh intersections. The latter may lead to geometrically ill-posed situations, which may translate into bad conditioning of the resulting problem when elliptic equations are considered. Several approaches have emerged over the years involving either the use of appropriate stabilization operators, cell aggregation, or the use of level sets to redefine the finite element spaces \cite{cutfem15,BADIA2018533,LEHRENFELD2016716,Oyarzua19,XIAO2020112964,Lozinski2019,phi23}. Many of these approaches allow high-degree polynomial approximations as well.

In this paper, we focus on a numerical discretization strategy for handling domain boundary conditions introduced relatively recently by G. Scovazzi and his collaborators: the shifted boundary method (SBM). The SBM was initially proposed by Main \& Scovazzi in \cite{Main2018a,Main2018b} for solving boundary value problems. The main idea behind the SBM is first to relocate -- or \textit{shift} -- the domain boundary from the \textit{true} to an approximate -- or \textit{surrogate} -- boundary. The latter is composed of appropriately defined mesh faces. The subset of mesh elements within the surrogate boundary constitutes the surrogate domain on which the discrete solution is computed. To ensure that the boundary conditions are imposed with the desired accuracy, the second idea in the SBM is to modify the problem data to account for the geometrical shift of the boundary. This step is classically used as a shift operator -- from the surrogate to the true boundary -- provided by a truncated Taylor series development. This approach provides a simple, efficient, and robust method that overcomes complex meshing tasks, naturally incorporates the ability to deal with curved geometries, and avoids potential problems with small cut cells. The method has been successfully extended to hyperbolic and parabolic equations \cite{Song2018,Ciallella2022,atallah2023highorder,LI2024116461}, to the management of fixed and moving interfaces \cite{Li2020,Colomes2021,Carlier2022,atallah2023highorder}, to solid mechanics \cite{Atallah2021a,Li2021}, and to higher orders \cite{Nouveau2019,Atallah2022,Collins2023,Ciallella2022}.

% Conditioning
\underline{\textit{Accuracy and conditioning}}. 
Two key elements in most of the unfitted finite element approaches mentioned are: i) the way in which the boundary conditions are enforced and ii) the polynomial approximation used to achieve such enforcement. In the context of the SBM, the data is typically extrapolated from the surrogate boundary to the true boundary. For the original formulation, this is done in terms of truncated Taylor series, which -- as shown in \cite{Ciallella2022} -- is exactly equivalent to the evaluation of the cell polynomials onto the true boundary. This leaves several options, as the true boundary may fall within the surrogate domain or outside. The latter approach is usually the one considered in most SBM publications, with the exception of the recent work considering a mix of the two combined with adaptive mesh refinement \cite{yang2023optimal}. When considering polynomials of high degree, this aspect may have a strong impact on the accuracy and stability of SBM. Extrapolation of such polynomials is known to converge less optimal and lead to less well-conditioned algebraic equations \cite{Badia2022}. 

Finally, the form of the terms included in the variational statement may have an impact on the accuracy and conditioning of the problem. In many of the works mentioned on high-order unfitted finite elements, including the SBM, the boundary condition is enforced weakly. For Dirichlet conditions, this is often achieved using Nitsche's variation formulation \cite{Nitsche1971}. To impose more general conditions, generalized variants of Nitsche's method could be used, as those proposed in \cite{Benzaken2022,Juntunen2009}. To our knowledge, the only example of this in the context of unfitted finite elements is in \cite{WINTER2018220}. These formulations involve several additional terms, including adjoint consistent corrections and penalization terms. Another related and classical approach to enforce Dirichlet conditions weakly is the one proposed by Aubin, also known as a penalty method \cite{Aubin1970}. This approach is somewhat simpler compared to  Nitsche's method. However, it is not adjoint-consistent, which may lead to sub-optimal convergence, e.g., see \cite{Douglas2002} and references therein.

\subsection{Paper aim and contribution}

In this paper, we propose and evaluate an SEM-SBM approach with general weak Robin boundary conditions. The first objective is -- besides providing a high-order approximation of genuine Robin conditions -- to obtain formulations allowing the consideration of spatially varying conditions going from Dirichlet to Neumann. Secondly, for problems involving small perturbations of one of the above limits, we would like the discrete approximation to be consistent with the asymptotics underlying the continuous PDE. These objectives will be achieved by combining the SBM with appropriate generalized weak forms \cite{Juntunen2009}. Finally, we perform a thorough study of the impact of the polynomial approximation on the surrogate boundary and of the variational form used on the numerical convergence and condition number. To this end, we study three variants of the SBM involving different maps from the mesh to the surrogate boundary and generalized weak forms of both the Nitsche and Aubin types.

\subsection{Structure of paper}

The remainder of this paper is outlined as follows: In Section \ref{sec:mathematical_problem}, the mathematical problem is stated for the two-dimensional scalar Poisson problem in strong form with Dirichlet, Neumann, and Robin boundary data. Then, in Section \ref{sec:prelimiaries}, we outline various preliminaries for the proposed model, including the computational domains, mappings, finite element tessellation, and more. In Section \ref{sec:numerical_method_and_discretization}, the numerical method -- in terms of the SEM -- is outlined on conformal meshes. This section is followed by Section \ref{sec:poisson_sbm}, where the polynomial corrected shifted boundary is introduced in the Dirichlet and Neumann context. Hereafter, in Section \ref{sec:cons_robin}, the case of Robin conditions are handled. In Section \ref{sec:numerical_results}, we present multiple numerical experiments, including convergence studies and numerical analysis of matrix conditioning. Lastly, in Section \ref{sec:conclusion}, a conclusion is formed to summarize the main findings and achievements of the paper.
 \section{Mathematical problem}\label{sec:mathematical_problem}

First, we define the mathematical problem by considering a two-dimensional ($d=2$) planar domain in a Euclidean space, $\Omega \subset {\rm I\!R}^{d}$, with a potentially curved Lipschitz boundary, $\Gamma = \partial \Omega$. On the boundary, we define the outward-facing normal and tangent unit vectors by $\boldsymbol{n} = \boldsymbol{n}(x,y): \Gamma \mapsto {\rm I\!R}^{d}$ and $\boldsymbol{t} = \boldsymbol{t}(x,y): \Gamma \mapsto {\rm I\!R}^{d}$, respectively. See Figure \ref{fig:computational_domains} for a visualization. In $\Omega$, we consider the scalar solution, $u = u(x,y): \Omega \mapsto {\rm I\!R}$, of the following elliptic problem with different non-homogeneous boundary conditions, where $u \in C^2(\Omega)$, such that
\begin{equation}\label{eq:Poisson}
    \begin{split}
        -\boldsymbol{\nabla}^2 u  +\alpha u                              &= f, \quad \text{in} \quad \Omega, \\
        u                                                   &= u_D, \quad \text{on} \quad \Gamma_D, \\
        \boldsymbol{\nabla} u \cdot \boldsymbol{n}                       &= q_N, \quad \text{on} \quad \Gamma_N,\\
        u + \varepsilon ( \boldsymbol{\nabla} u \cdot \boldsymbol{n} )   &=  u_{RD} + \varepsilon q_{RN},  \quad \text{on} \quad \Gamma_R.
        \end{split}
\end{equation}
Here $\Gamma$ is partitioned into boundaries for Dirichlet ($\Gamma_D$), Neumann ($\Gamma_N$), and Robin data ($\Gamma_R$) such that $\Gamma_D \cup \Gamma_N \cup \Gamma_R = \Gamma$. Furthermore, $\alpha \geq 0$ governs the reaction term to ensure a well-posed pure Neumann problem when $\Gamma_D = \Gamma_R = \emptyset$. Also, $f = f(x,y): \Omega \mapsto {\rm I\!R}$ is a given forcing function. The Dirichlet and Neumann data are denoted by $u_D = u_D(x,y): \Gamma_D \mapsto {\rm I\!R}$ and $q_N = q_N(x,y): \Gamma_N \mapsto {\rm I\!R}$, respectively. Moreover, the Dirichlet and Neumann contributions to the Robin boundary condition are denoted $u_{RD} = u_{RD}(x,y): \Gamma_R \mapsto {\rm I\!R}$ and $q_{RN} = q_{RN}(x,y): \Gamma_R \mapsto {\rm I\!R}$, respectively. Note that, in practice, only one of the two is known. Considering both terms, however, is interesting in view of the study of the asymptotic limits.

\subsection{Formal asymptotics of the full Robin problem}
    
Let us focus for a moment on the Robin problem obtained when $\Gamma_R = \Gamma$ as
\begin{equation}\label{eq:full_robin}
\begin{split}
        -\boldsymbol{\nabla}^2 u  +\alpha u                             &= f, \quad \text{in} \quad \Omega, \\
        u + \varepsilon ( \boldsymbol{\nabla} u \cdot \boldsymbol{n} )  &=  u_{RD} + \varepsilon q_{RN},  \quad \text{on} \quad \Gamma.  
        \end{split}
\end{equation}
Interestingly, the Robin condition can be used to mimic Dirichlet and Neumann conditions as a \textit{all-in-one} boundary formulation by choosing either small or large values of $\varepsilon$. This fact has already been exploited in \cite{refId0} to impose variable conditions on conformal meshes. Here, we try to go a little further. Indeed, as recalled in the introduction, many applications involve small perturbations of a Dirichlet or Neumann condition. The cases are obtained when $\varepsilon\ll1$ or $\varepsilon^{-1}\ll1$ -- but not zero -- in \eqref{eq:full_robin}, respectively. Using classical formal asymptotics, we can show the following. 

\begin{proposition}[Robin-Dirichlet asymptotics] For a smooth boundary $\Gamma$, and $u_{RD}$ and $q_{RN}$ smooth enough functions, the full Robin problem in \eqref{eq:full_robin} admits the following formal asymptotics for $\varepsilon\rightarrow 0$ as an series expansion as
\begin{equation}\label{eq:robin_D_expansion}
 u = u_0 + \varepsilon u_1 +\sum_{m > 1}\varepsilon^m u_m,
 \addtocounter{equation}{1} \addtocounter{lettercounter}{1}\tag{\theequation \alph{lettercounter}}
\end{equation}
where the leading order term, $u_0 = u_0(x,y): \Omega \mapsto {\rm I\!R}$, verifies the elliptic equation given Dirichlet boundary data as
\begin{equation}\label{eq:Robin_D_O0}
        -\boldsymbol{\nabla}^2 u_0  + \alpha u_0  = f, \quad \text{in} \quad \Omega \quad \text{with} \quad
        u_0  = u_{D,0}:= u_{RD} \quad \text{on} \quad \Gamma,  \addtocounter{equation}{0} \addtocounter{lettercounter}{1}\tag{\theequation \alph{lettercounter}}
\end{equation}
while the Dirichlet problems give the first and higher-order corrections $u_m = u_m(x,y): \Omega \mapsto {\rm I\!R}$ for $\forall m\geq1$ as
\begin{equation}\label{eq:Robin_D_O1}
        -\boldsymbol{\nabla}^2 u_1  + \alpha u_1  = 0, \quad \text{in} \quad \Omega \quad \text{with} \quad
        u_1 =  u_{D,1}:= q_{RN}-\partial_n u_0 \quad \text{on} \quad \Gamma, 
        \addtocounter{equation}{0} \addtocounter{lettercounter}{1}\tag{\theequation \alph{lettercounter}}
\end{equation}
and $\forall m > 1$
\addtocounter{equation}{0}
\begin{equation}\label{eq:Robin_D_Om}
        -\boldsymbol{\nabla}^2 u_m + \alpha u_m = 0, \quad \text{in} \quad \Omega \quad \text{with} \quad
        u_m = u_{D,m}:= -\partial_n u_{m-1} \quad \text{on} \quad \Gamma,
         \addtocounter{equation}{0} \addtocounter{lettercounter}{1}\tag{\theequation \alph{lettercounter}}
\end{equation}
where $\partial_n$ denotes the derivative with respect to the normal direction of the boundary.
\end{proposition}

\begin{proposition}[Robin-Neumann asymptotic] For a smooth boundary $\Gamma$, and $u_{RD}$ and $q_{RN}$ smooth enough functions, the full Robin problem in \eqref{eq:full_robin} admits the following formal asymptotics for $\varepsilon^{-1}\rightarrow 0$
\begin{equation}\label{eq:robin_N_expansion}
 u = u_0 + \varepsilon^{-1} u_1 +\sum_{m >1}\varepsilon^{-m} u_m,  \addtocounter{equation}{1} \setcounter{lettercounter}{1}\tag{\theequation \alph{lettercounter}}
\end{equation}
where the leading order term $u_0 = u_0(x,y): \Omega \mapsto {\rm I\!R}$ verifies the elliptic equation given Neumann boundary data as
\begin{equation}\label{eq:Robin_N_O0}
        -\boldsymbol{\nabla}^2 u_0        + \alpha u_0                             = f, \quad \text{in} \quad \Omega\quad\text{with}\quad
        \partial_nu_0  = q_{N,0}:= q_{RN} \quad \text{on} \quad \Gamma,  \addtocounter{equation}{0} \addtocounter{lettercounter}{1}\tag{\theequation \alph{lettercounter}}
\end{equation}
while the Neumann problems give the first and higher-order corrections $u_m = u_m(x,y): \Omega \mapsto {\rm I\!R}$ for $\forall m\geq1$ as
\begin{equation}\label{eq:Robin_N_O1}
        -\boldsymbol{\nabla}^2 u_1     +\alpha u_1                                = 0, \quad \text{in} \quad \Omega\quad\text{with}\quad
        \partial_nu_1 =  q_{N,1} = u_{N,1} := u_{RD} - u_0 \quad \text{on} \quad \Gamma,    \addtocounter{equation}{0} \addtocounter{lettercounter}{1}\tag{\theequation \alph{lettercounter}}
\end{equation}
and $\forall m > 1$
\begin{equation}\label{eq:Robin_N_Om}
        -\boldsymbol{\nabla}^2 u_m            +\alpha u_m                     = 0, \quad \text{in} \quad \Omega\quad\text{with}\quad
        \partial_nu_m =  q_{N,m} = u_{N,m} := - u_{m-1} \quad \text{on} \quad \Gamma.  \addtocounter{equation}{0} \addtocounter{lettercounter}{1}\tag{\theequation \alph{lettercounter}}
\end{equation}
\end{proposition}

The above propositions can be proven -- for smooth enough data and domain -- using classical matched asymptotic expansions in $\varepsilon$ or $\varepsilon^{-1}$, e.g., see \cite{murray84} chapter 7 or \cite{zeytounian} chapter 3. We omit the details for brevity. Moreover, we will say that a numerical approximation of \eqref{eq:full_robin} is asymptotic preserving (AP) if the discrete asymptotics of such approximation is consistent with the continuous ones at all orders in powers of $\varepsilon$. Note that when the Robin condition is associated with a physical process, usually $q_{RN}=0$, for the small $\varepsilon$ limit, or $u_{RD}=0$ in the small $\varepsilon^{-1}$ limit. However, keeping these values allows us to include the discussion of when using Robin conditions to handle Dirichlet or Neumann boundaries.
 \section{Domain discretization: mesh and boundary mapping}\label{sec:prelimiaries}

\begin{figure}[t]
     \centering
     \begin{subfigure}
         \centering
         \includegraphics{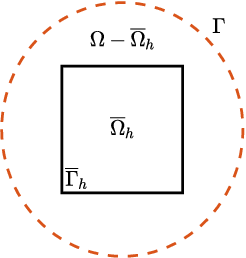}
     \end{subfigure}
     \hspace{0.5cm}
     \begin{subfigure}
         \centering
         \includegraphics{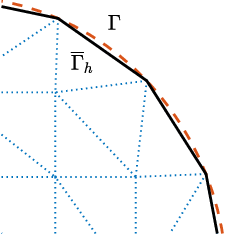}
     \end{subfigure}
     \vskip \baselineskip
     \begin{subfigure}
         \centering
         \includegraphics{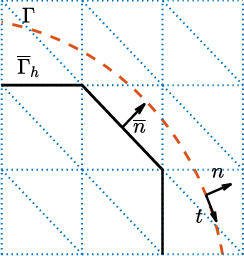}
     \end{subfigure}
     \hspace{0.5cm}
     \begin{subfigure}
         \centering
         \includegraphics{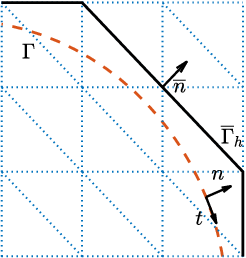}
     \end{subfigure}
        \caption{Conceptual layout of the \textit{true} domain and the \textit{surrogate} domain (top left). Conformal affine meshing (top right). \textit{Extrapolation} mode (bottom left) and \textit{interpolation} mode (bottom right). Surrogate domain boundary (solid black): $\overline{\Gamma}_h$, true domain boundary (dashed red): $\Gamma$, and triangulated mesh (dotted blue in the background): elements $\Omega^n$.}
        \label{fig:computational_domains}
\end{figure}

In this section, different domain discretizations will be defined to form the geometrical understanding and notation. For this, we consider regular (quasi-uniform) tessellations using $N_{\text{elm}}$ conforming non-overlapping affine triangular elements. The computational domain, $\Omega_h$, is thus defined by 
\begin{equation}
     \Omega_h  = \bigcup_{n = 1}^{N_{\text{elm}}} \Omega^n,
\end{equation}
with $\Omega^n$ being the $n$'th element. We denote the maximum edge length for each element as $h_{\max}^n$. Similar local quantities, such as the minimum, $h_{\min}^n$, and the average, $h_{\text{avg}}^n$, edge length can be computed. The associated global measure of the maximum edge length is taken as $h_{\max} = \max (\{h_{\max}^n\}^{N_{\text{elm}}}_{n=1}) $, where $h_{\min}$ and $h_{\text{avg}}$ are found similarly. 

Let $\Gamma_h$ denote the boundary of the computational domain. We will consider two main cases in this paper: a conformal (fitted) and a non-conformal (unfitted) boundary approach. In the fitted case, $\Gamma_h$ is conformal with the analytical boundary $\Gamma$, and is usually defined by a classical iso-parametric high-degree finite element approximation. The analytical boundary is embedded into a non-conformal background mesh in the unfitted case. Here, we distinguish between the \textit{true} domain, defined by the analytical boundary $\Gamma$, and the \textit{surrogate} domain. The surrogate domain is denoted by $\overline{\Gamma}_h$. It is defined here using level set theory \cite{Osher2001} equipped with a signed distance function -- such as the Euclidean metric -- to capture the geometry of the domain implicitly.
% $\mathcal{L} = \mathcal{L}(\Omega)$ -- --, where the upper, lower, and zero level set is defined as
% \begin{equation}
% \begin{split}
%     \mathcal{L}_{+} &= \{ (x,y) | \mathcal{L}(\Omega) > 0 \},\quad \text{for} \quad (x,y)	\subset \Omega, \\
%     \mathcal{L}_{-} &= \{ (x,y) | \mathcal{L}(\Omega) < 0 \},\quad \text{for} \quad (x,y)	\not \subset \Omega, \\
%     \mathcal{L}_{0} &= \{ (x,y) | \mathcal{L}(\Omega) = 0 \},\quad \text{for} \quad (x,y)    \subset \Gamma. 
%     \end{split}
% \end{equation} 

For a given element, $\Omega^n$, of arbitrary type, three situations can occur depending on the level set sign at its spanning vertices: inside $\Omega$ (all positive signs), outside $\Omega$ (all negative signs), crossed by the true boundary $\Gamma$ (mix of signs). Once the elements crossed by the true boundary are known, the set of element edges closest to the true boundary can be located using a closest-point map (or some other projection strategy). Ultimately, we obtain a set of edges defining the surrogate boundary, $\overline{\Gamma}_h$. This implicitly defines a surrogate domain $\overline{\Omega}_h$ composed of the elements contained within the surrogate boundary. As before, we denote these elements by $\{\Omega^n\}_{n=1}^{N_{\text{elm}}}$ .
% Define n dot n
On $\overline{\Gamma}_h$ we define the outward-facing normal unit vector by $\overline{\boldsymbol{n}} = \overline{\boldsymbol{n}}(x,y): \overline{\Gamma}_h \mapsto {\rm I\!R}^{d}$. 
% If possible, we slightly modify the construction of the surrogate boundary to make sure that $\boldsymbol{n} \cdot \overline{\boldsymbol{n}} \geq 0$. Hereby, the natural orientation of the boundary is retained. Note that this is not strictly necessary for the stability of the method as shown in, e.g.,  \cite{Atallah2020a,Atallah2020b}. With these criteria, constraints are set on the feasible scale of the true geometrical features compared to the size of the elements.

\subsection{Boundary mapping}\label{sec:mappings}

For the purposes of this paper, we consider two different ways of defining the surrogate boundary $\overline{\Gamma}_h$. These are referred to as \textit{extrapolation} and \textit{interpolation} mode. These are defined as follows
\begin{itemize}
    \item Extrapolation: $\overline{\Omega}_h$ is constructed using only whole elements inside the true domain $\Omega$.
    \item Interpolation: $\overline{\Omega}_h$ is constructed by including all whole elements inside $\Omega$ and the ones intersecting $\Gamma$.
\end{itemize}

To support the definitions, the true and different surrogate domains are illustrated in a Cartesian coordinate system spanned by the horizontal $x$-axis and the vertical $y$-axis in Figure \ref{fig:computational_domains}. Also, an example of a conformal affine mesh is visualized in the figure. Interpolation-based embedded methods have been studied in recent works, see, e.g., \cite{yang2023optimal,Badia2022}, as well as \cite{yang2023optimal} in the context of the second-order SBM approximations on adaptive quadrilateral meshes. In this work, we consider the latter approach, combined with very high-order spectral element approximations on simplicial meshes, although the exact same formulations can be used on tensor product Cartesian meshes.

\begin{remark}
In the interpolation mode, the discrete domain includes a small region outside the analytical one. This requires the definition of an extension of the problem data and possibly of the exact solution in a small region outside the physical domain. This is a well-known issue related to the construction of embedded methods (see, e.g.,  \cite{Glowinski1994,Burman2016,Lozinski2019,Larson2019,Lehrenfeld2018}). Some practical solutions have been proposed in the past, e.g., in \cite{Glowinski1994,Fabreges2013}. We do not dwell on this issue in this work. All the applications considered involve manufactured solutions, for which this extension is naturally obtained from the (known) analytical solution.
\end{remark}

% Mapping
Once the surrogate boundary is defined, a key ingredient is the availability of a mapping from the surrogate to the true boundary, where the boundary conditions are defined. This mapping is denoted by $\mathcal{M}$ as
\begin{equation}\label{eq:Gamma-map}
     \mathcal{M}(\overline{\boldsymbol{x}}): \overline{\Gamma}_h \mapsto \Gamma, \quad \quad \quad \overline{\boldsymbol{x}} \mapsto \boldsymbol{x},
\end{equation}
where a point on $\Gamma$ is denoted by $\boldsymbol{x} = (x,y)$ and similarly on $\overline{\Gamma}_h$ by $\boldsymbol{\overline{x}} = (x,y)$. Using the above, we have that $\boldsymbol{x} = \mathcal{M}(\boldsymbol{\overline{x}})$, i.e., the set of points on the true boundary spanned by the points of $\overline{\Gamma}_h$. This allows us to define a distance vector as $\boldsymbol{d} = \boldsymbol{d}(\boldsymbol{\overline{x}}) = \boldsymbol{x} - \boldsymbol{\overline{x}}$. Note that there are several possible ways to define the mapping. For example, with the closest-point projection, any point on $\overline{\Gamma}$ is mapped to $\Gamma$ using the true normal, $\boldsymbol{n}$, such that $\boldsymbol{d} \cdot \boldsymbol{n} \equiv 1$ and hereby $\boldsymbol{d} = \|\boldsymbol{d}\| \boldsymbol{n}$. Other definitions are explored in this paper. The mapping, $\mathcal{M}$, allows us to define extensions of any particular function from the surrogate to the true boundary (and vice-versa). 
% as
% \begin{equation}
%     \overline{g}(\overline{\boldsymbol{x}}) \equiv g(\mathcal{M}(\overline{\boldsymbol{x}})) = g(\boldsymbol{x}).
% \end{equation}
 \section{Spectral element discretization on conformal meshes}\label{sec:numerical_method_and_discretization}

In the following, we consider and define the spectral element method on conformal meshes for generalized Aubin's and Nitsche's methods. This is then related to the asymptotic preserving properties of the two variational forms.

\subsection{Spectral element expansion}\label{sec:SEM}

We consider approximate solutions in the finite element space of continuous and piece-wise polynomials of degree $P$ as
\begin{equation}\label{eq:VP}
    \mathcal{V}^P_h = \{v \in C^0(\Omega); \forall n \in \{1,..., N_{\text{elm}} \}, v|_{\Omega^n} \in {\rm I\!P}^{P}\}.
\end{equation}
Thus, the discrete unknown, $u_h$, can be expressed as $N_{\text{elm}}$ local element-based solutions, $u^n_h = u^n_h(x,y): \Omega^n \mapsto {\rm I\!R}$, as
\begin{equation}\label{eq:sol_rep_1}
    u_h = u_h(x,y) = \bigoplus_{n = 1}^{N_{\text{elm}}} u^n_h(x,y), \quad \text{for}  \quad (x,y) \in \Omega^n, 
\end{equation}
In this work, $u_h$ is continuous and expanded on local bases whose restriction to an element $\Omega^n$ may be equivalently defined in terms of nodal unknowns or modal expansions as in \cite{Karniadakis2005}
\begin{equation}\label{eq:sol_rep_2}
    u^n_h(x,y) \approx \sum_{m=1}^{N_{\text{ep}}} \hat{u}^n_{m} \psi_{m}(x,y)=\sum_{i=1}^{N_{\text{ep}}} u^n_ih_{i}(x,y), \quad \text{for}  \quad (x,y) \in \Omega^n,
\end{equation}
where $N_{\text{ep}}$ is the number of quadrature element points in an element. The local modal form is expressed as a sum of $N_{\text{ep}}$ local expansion coefficients, $\hat{u}^n_{m}$, and associated modal two-dimensional basis functions, $\psi_{m}$, from the family of Jacobi polynomials up to order $P$ defined on the $N_{\text{ep}}$ quadrature points. Moreover, the local nodal counterpart is expressed as a sum of $N_{\text{ep}}$ local solution values, $u^n_i$, and nodal two-dimensional basis functions, $h_i$, as Lagrange polynomials up to order $P$ defined on the $N_{\text{ep}}$ quadrature points. For completeness, we denote the $N_{\text{ep}}$ quadrature points in $\Omega^n$ as $\{x_i,y_i\}_{i=1}^{N_{\text{ep}}}$, such that $u_i^n = u^n(x_i,y_i)$ and the Lagrange polynomials have the cardinal property $h_j(x_i,y_i) = \delta_{ij}$. In the case of finite triangular elements, the number of element points, $N_{\text{ep}}$, is linked to the polynomial order of the basis functions, $P$, via $N_{\text{ep}} =\frac{(P+1)(P+2)}{2}$. 

\begin{figure}[t]
    \centering
    \includegraphics{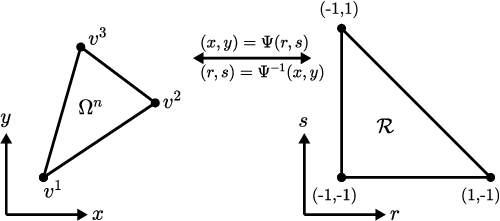}
    \caption{Concept of the local element domain, $\Omega^n$, the reference domain, $\mathcal{R}$, and the affine mapping, $\Psi: \Omega^n \mapsto \mathcal{R}$.}
    \label{fig:affine_mapping}
\end{figure}

% Define the reference domain:
In practice, we adopt an iso-parametric approach, where each physical element, $\Omega^n$, is mapped onto a triangular reference domain, $\mathcal{R} = \{\boldsymbol{r} = (r,s) | (r, s) \geq -1 ; r+s \leq 0 \}$, by introducing the affine mapping, $\Psi = \Psi(r,s): \Omega^n \mapsto \mathcal{R}$, such that $(x,y) = \Psi(r,s)$ and $(r,s) = \Psi^{-1}(x,y)$. See Figure \ref{fig:affine_mapping} for a visualization. In accordance with \cite{Hesthaven2007}, we construct an orthonormal basis, $\psi_m(r,s)$, on a set of quadrature points, $(r,s)$, in $\mathcal{R}$. The basis functions are chosen as the family of Jacobi polynomials. Moreover, the quadrature points are chosen to be Gauss-Lobatto distributed by applying an optimized blending and wrapping procedure of the nodes by \cite{Warburton2006} such that the Lebesque constant is minimized. This delicate arrangement ultimately leads to a well-behaved generalized Vandermonde matrix for interpolation as
\begin{equation}
   \mathcal{V} = \mathcal{V}_{ij} = \psi_j(r_i,s_i), \quad \text{for} \quad (i,j) = 1,...,N_{\text{ep}}, \quad \mathcal{V} \in {\rm I\!R}^{N_\text{ep} \times N_\text{ep}}.
\end{equation}

\subsection{Variational formulation on conformal meshes: a generalized Nitsche's method}

We consider a weak form of \eqref{eq:Poisson} based on a generalized Nitsche's method, hereby embedding all types of boundary conditions. On conforming meshes, the method's consistency, coercivity, and error estimates have been widely studied in the past, e.g., see \cite{Nitsche1971,Benzaken2022,Juntunen2009}. The classical $L^2$ scalar product over the set $A$ is denoted by 
\begin{equation}\label{eq:scalar} 
    (f,g)_{A}  =\int\limits_{A} f\,g\, dA.
\end{equation}

The weak formulation is now to find $u_h \in \mathcal{V}^P_h$ such that $\forall v_h \in\mathcal{V}^P_h$, where the following holds 
\begin{equation}\label{eq:variational_full_cbm} 
\begin{split}
        (\boldsymbol{\nabla} u_h,\boldsymbol{\nabla} v_h)_{\Omega_h} +&\alpha (u_h, v_h)_{\Omega_h} - ( f,v_h)_{\Omega_h}-(q_N,v_h)_{\Gamma_{h,N}}\\
        - &(\boldsymbol{\nabla} u_h\cdot \boldsymbol{n},v_h)_{\Gamma_{h,D}} -  (u_h-u_D,\boldsymbol{\nabla} v_h\cdot \boldsymbol{n})_{\Gamma_{h,D}}+ \gamma^{-1} (u_h-u_D,v_h)_{\Gamma_{h,D}}  \\
        -&\dfrac{\gamma}{\gamma+\varepsilon} (\boldsymbol{\nabla} u_h\cdot \boldsymbol{n},v_h )_{\Gamma_{h,R}}
        - \dfrac{\gamma}{\gamma+\varepsilon} ( u_h-u_{RD},\boldsymbol{\nabla} v_h\cdot \boldsymbol{n})_{\Gamma_{h,R}}+
        \dfrac{1}{\gamma+\varepsilon} (u_h-u_{RD},v_h)_{\Gamma_{h,R}}\\
        -&\dfrac{\varepsilon}{\gamma+\varepsilon} (q_{RN},v_h)_{\Gamma_{h,R}} -
        \dfrac{\varepsilon\gamma}{\gamma+\varepsilon} (\boldsymbol{\nabla} u_h\cdot \boldsymbol{n}-q_{RN},\boldsymbol{\nabla} v_h\cdot \boldsymbol{n})_{\Gamma_{h,R}} = 0.
\end{split}
\end{equation}
The variational form above, which includes several stabilizing and symmetrizing terms, has been analyzed in depth in the past and shown to be continuous and coercive \cite{Nitsche1971,Benzaken2022,Juntunen2009}. In particular, the parameter $\gamma^{-1}$ is the penalty coefficient of the classical Dirichlet Nitsche method and of the generalized Nitsche method of \cite{Juntunen2009}. These have been studied in several past works, e.g., see \cite{Prenter2018,Benzaken2022,Juntunen2009} and references therein.

With some algebra, the Euler-Lagrange equations -- and its strong consistency -- can be shown to be
\begin{equation}\label{eq:variational_full_EL} 
\begin{split}
        ( -\boldsymbol{\nabla}^2 u_h + \alpha u_h-f,v_h)_{\Omega_h}    +  & (u_h -u_D, \gamma^{-1} v_h- \boldsymbol{\nabla} v_h\cdot \boldsymbol{n})_{\Gamma_{h,D}}+ (\boldsymbol{\nabla} u_h\cdot \boldsymbol{n}-q_N,v_h)_{\Gamma_{h,N}} \\
        + &\dfrac{1}{\gamma+\varepsilon} (u_h-u_{RD} + \varepsilon \boldsymbol{\nabla} u_h\cdot \boldsymbol{n}- \varepsilon q_{RN},v_h-\gamma \boldsymbol{\nabla} v_h\cdot \boldsymbol{n})_{\Gamma_{h,R} }
        = 0.
\end{split}
\end{equation}

To simplify the notation, from now on, we will omit the subscript "$h$", and assume implicitly that all quantities involving discrete polynomial functions are expressed and integrated on mesh entities. Also, we recast the variational form \eqref{eq:variational_full_cbm} in the short notation as
\begin{equation}
    a \langle u,v \rangle = b\langle v \rangle,  \addtocounter{equation}{1} \setcounter{lettercounter}{1} \tag{\theequation \alph{lettercounter}} \label{eq:weak_cbm} 
\end{equation}
where
\begin{equation}
    a\langle u,v \rangle = a\langle u,v \rangle_{\Omega} + a\langle u,v \rangle_{\Gamma_D} + a \langle u,v \rangle_{\Gamma_N} + a \langle u,v \rangle_{\Gamma_R},  \addtocounter{equation}{0} \addtocounter{lettercounter}{1}\tag{\theequation \alph{lettercounter}} \label{eq:weak_cbm_bilinear} 
\end{equation}
\begin{equation}
    b\langle v \rangle = b\langle v \rangle_{\Omega} + b\langle v \rangle_{\Gamma_D} + b \langle v \rangle_{\Gamma_N} + b\langle v \rangle_{\Gamma_R},   \addtocounter{equation}{0} \addtocounter{lettercounter}{1}\tag{\theequation \alph{lettercounter}}  \label{eq:weak_cbm_linear} 
  \end{equation}
with
\begin{align}
    a\langle u,v \rangle_{\Omega} &= (\boldsymbol{\nabla} u, \boldsymbol{\nabla} v)_{\Omega} + \alpha  ( u,  v)_{\Omega},  \addtocounter{equation}{0} \addtocounter{lettercounter}{1}\tag{\theequation \alph{lettercounter}} \label{eq:weak_cbm_bilinear_2d}\\ 
    a\langle u,v \rangle_{\Gamma_D} &= \gamma^{-1} (u,v)_{\Gamma_D} - (\boldsymbol{\nabla} u \cdot \boldsymbol{n},v)_{\Gamma_D} -(u,\boldsymbol{\nabla} v \cdot \boldsymbol{n})_{\Gamma_D},  \addtocounter{equation}{0} \addtocounter{lettercounter}{1}\tag{\theequation \alph{lettercounter}} \label{eq:weak_cbm_bilinear_dbc}\\
    a \langle u,v \rangle_{\Gamma_N} &= 0, \addtocounter{equation}{0} \addtocounter{lettercounter}{1}\tag{\theequation \alph{lettercounter}} \label{eq:weak_cbm_bilinear_nbc}\\
    a \langle u,v \rangle_{\Gamma_R} &= \frac{1}{\varepsilon + \gamma} (u,v)_{\Gamma_R} 
    - \frac{\gamma}{\varepsilon + \gamma} (\boldsymbol{\nabla} u \cdot \boldsymbol{n},v)_{\Gamma_R}
    - \frac{\gamma}{\varepsilon + \gamma} (u,\boldsymbol{\nabla} v \cdot \boldsymbol{n})_{\Gamma_R}
    - \frac{\varepsilon\gamma}{\varepsilon + \gamma} (\boldsymbol{\nabla} u \cdot \boldsymbol{n},\boldsymbol{\nabla} v \cdot \boldsymbol{n})_{\Gamma_R} ,  \addtocounter{equation}{0} \addtocounter{lettercounter}{1}\tag{\theequation \alph{lettercounter}}  \label{eq:weak_cbm_bilinear_rbc}
\end{align}
\begin{align} 
    b\langle v \rangle_{\Omega} &= (f,v)_{\Omega},  \addtocounter{equation}{0} \addtocounter{lettercounter}{1}\tag{\theequation \alph{lettercounter}} \label{eq:weak_cbm_linear_2d}\\
    b\langle v \rangle_{\Gamma_D} &= \gamma^{-1} (u_D,v)_{\Gamma_D} 
    -(u_D,\boldsymbol{\nabla} v \cdot \boldsymbol{n})_{\Gamma_D} ,  \addtocounter{equation}{0} \addtocounter{lettercounter}{1}\tag{\theequation \alph{lettercounter}} \label{eq:weak_cbm_linear_dbc}\\
    b \langle v \rangle_{\Gamma_N} &= (q_N,v)_{\Gamma_N},  \addtocounter{equation}{0} \addtocounter{lettercounter}{1}\tag{\theequation \alph{lettercounter}}  \label{eq:weak_cbm_linear_nbc}\\
    b\langle v \rangle_{\Gamma_R} &=  \frac{1}{\varepsilon + \gamma} (u_{RD},v)_{\Gamma_R} + \frac{\varepsilon}{\varepsilon + \gamma} ( q_{RN},v)_{\Gamma_R}
     - \frac{\gamma}{\varepsilon + \gamma} (u_{RD},\boldsymbol{\nabla} v \cdot \boldsymbol{n})_{\Gamma_R}
    - \frac{\varepsilon\gamma}{\varepsilon + \gamma} (q_{RN},\boldsymbol{\nabla} v \cdot \boldsymbol{n})_{\Gamma_R}.  \addtocounter{equation}{0} \addtocounter{lettercounter}{1}\tag{\theequation \alph{lettercounter}}  \label{eq:weak_cbm_linear_rbc}
\end{align}

\subsection{Variational formulation on conformal meshes: a generalized Aubin's method}

We now consider a reduced form of Nitche's formulation essentially generalizing the method by Aubin \cite{Aubin1970}. By inspecting \eqref{eq:variational_full_EL}, we can see that the Robin contributions can be written as a linear combination of a Neumann penalty with coefficient $\varepsilon/(\varepsilon+\gamma)$, a Nitche penalty with penalty parameter $\gamma^{-1}$, and coefficient $\gamma/(\varepsilon+\gamma)$. To obtain a method similar to Aubin's penalty formulation, we consider the same above combination without the symmetrizing terms in the penalties. In the short-form notation, this gives the generalized Aubin variational form as
\begin{align}
    a\langle u,v \rangle_{\Omega} &= (\boldsymbol{\nabla} u, \boldsymbol{\nabla} v)_{\Omega} + \alpha  ( u,  v)_{\Omega},  \addtocounter{equation}{1} \setcounter{lettercounter}{1}\tag{\theequation \alph{lettercounter}} \label{eq:weak_cbm_bilinear_2d_aubin}\\ 
    a\langle u,v \rangle_{\Gamma_D} &= \gamma^{-1} (u,v)_{\Gamma_D} - (\boldsymbol{\nabla} u \cdot \boldsymbol{n},v)_{\Gamma_D},  \addtocounter{equation}{0} \addtocounter{lettercounter}{1}\tag{\theequation \alph{lettercounter}} \label{eq:weak_cbm_bilinear_dbc_2}\\
    a \langle u,v \rangle_{\Gamma_N} &= 0, \addtocounter{equation}{0} \addtocounter{lettercounter}{1}\tag{\theequation \alph{lettercounter}} \label{eq:weak_cbm_bilinear_nbc_aubin}\\
    a \langle u,v \rangle_{\Gamma_R} &= \frac{1}{\varepsilon + \gamma} (u,v)_{\Gamma_R} 
    - \frac{\gamma}{\varepsilon + \gamma} (\boldsymbol{\nabla} u \cdot \boldsymbol{n},v)_{\Gamma_R},  \addtocounter{equation}{0} \addtocounter{lettercounter}{1}\tag{\theequation \alph{lettercounter}}  \label{eq:weak_cbm_bilinear_rbc_aubin}
\end{align}
and
\begin{align} 
    b\langle v \rangle_{\Omega} &= (f,v)_{\Omega},  \addtocounter{equation}{0} \addtocounter{lettercounter}{1}\tag{\theequation \alph{lettercounter}} \label{eq:weak_cbm_linear_2d_aubin}\\
    b\langle v \rangle_{\Gamma_D} &= \gamma^{-1} (u_D,v)_{\Gamma_D},  \addtocounter{equation}{0} \addtocounter{lettercounter}{1}\tag{\theequation \alph{lettercounter}} \label{eq:weak_cbm_linear_dbc_aubin}\\
    b \langle v \rangle_{\Gamma_N} &= (q_N,v)_{\Gamma_N},  \addtocounter{equation}{0} \addtocounter{lettercounter}{1}\tag{\theequation \alph{lettercounter}}  \label{eq:weak_cbm_linear_nbc_aubin}\\
    b\langle v \rangle_{\Gamma_R} &=  \frac{1}{\varepsilon + \gamma} (u_{RD},v)_{\Gamma_R} + \frac{\varepsilon}{\varepsilon + \gamma} ( q_{RN},v)_{\Gamma_R}.  \addtocounter{equation}{0} \addtocounter{lettercounter}{1}\tag{\theequation \alph{lettercounter}}  \label{eq:weak_cbm_linear_rbc_aubin}
\end{align}
Integration by part yields the Euler-Lagrange equations -- that emphasizes the consistency of the formulation -- as
\begin{equation}
         ( -\boldsymbol{\nabla}^2 u + \alpha u-f,v)_{\Omega}    +   \gamma^{-1}(u -u_D, v)_{\Gamma_{D}}+ (\boldsymbol{\nabla} u\cdot \boldsymbol{n}-q_N,v)_{\Gamma_{N}}
          + \dfrac{1}{\varepsilon+\gamma} (u-u_{RD} + \varepsilon \boldsymbol{\nabla} u\cdot \boldsymbol{n}- \varepsilon q_{RN},v)_{\Gamma_{R}}
        = 0.
\end{equation}

Note that the coercivity of \eqref{eq:weak_cbm_bilinear_2d_aubin}-\eqref{eq:weak_cbm_bilinear_rbc_aubin} can be shown by classical techniques based on trace inequalities to bound the gradient term on the boundary, e.g., see  \cite{Aubin1970,Ern2004} and \cite[Chapter 37]{Ern2021}. As mentioned in other works, in the Dirichlet case, e.g., see \cite{Douglas2002}, the penalty method requires appropriate choices of the penalty constant, and is in general less consistent with respect to the symmetric and adjoint consistent Nitsche method. However, as we will see in Section \ref{sec:numerical_results}, very high orders of approximation on embedded boundaries and the presence of additional derivative terms have a strong impact on the conditioning, which makes this weak form of interest to consider in detail.

\subsection{Algebraic system of equations}

Finally, following the standard Galerkin spectral element approach, a system of equations is obtained by choosing the test function, $v$, to equal the nodal basis functions, $h$, outlined in Section \ref{sec:SEM}. Hereafter, we evaluate all the integrals using accurate numerical quadrature \cite{Patera1984,Karniadakis2005,Xu2018}. This ultimately leads to a large algebraic system of equations for the nodal unknowns as
\begin{equation}\label{eq:linear_system_of_equations}
    \mathcal{A} \boldsymbol{u} = \boldsymbol{b}, \quad \text{where} \quad \mathcal{A} \in {\rm I\!R}^{N_\text{DOF} \times N_\text{DOF} }, \quad \text{and} \quad (\boldsymbol{u},\boldsymbol{b}) \in {\rm I\!R}^{N_\text{DOF}}.
\end{equation}
with $N_{\text{DOF}}$ being the number of degrees of freedom (DOF) -- corresponding to the number of nodal points, expansion coefficients, and/or basis functions -- in the discrete scheme. The system matrix and vector are denoted by $\mathcal{A}$ and $\boldsymbol{b}$, respectively, corresponding to the weak statements' left- and right-hand sides. The discrete solution vector $\boldsymbol{u}$ has been obtained using a direct linear solver to invert $\mathcal{A}$. Contributions to $\mathcal{A}$ and $\boldsymbol{b}$ are obtained on a local element level with discrete operators for integration and differentiation as defined in \cite{Hesthaven2007}.

\subsection{Asymptotic preserving properties of the Robin variational forms} \label{sec:AP_Robin}
The boundary entry of most interest in \eqref{eq:variational_full_cbm} is the one related to the Robin terms. Because of their asymptotic properties, described by Proposition 2.1 and 2.2, these terms allow setting -- with a single implementation -- different types of boundary conditions by just choosing a different value of the $\varepsilon$ parameter. To the authors' knowledge, this approach has never been tested in the context of embedded methods and -- in any case -- not in the framework of the SBM.

For this to be possible, it is important that the discrete equations correctly mimic the asymptotic behavior of the continuous problem. In this case, we speak of asymptotic preserving approximations. It has already been remarked in \cite{Juntunen2009} that the classical implementation of the Robin conditions (based on a modified Neumann flux) does not lend itself to this purpose. They show numerically that their proposed method is well posed in both the zeroth of $\varepsilon$ and $\varepsilon^{-1}$. Here, we go further and prove the asymptotic consistency of the method with the problems defined by Proposition 2.1 and 2.2.

\subsubsection{The Nitsche-Robin method}

First, recall that we solely consider the Robin boundary, such that $\Gamma = \Gamma_R$. Then, we introduce the following notation
\begin{equation}\label{eq:notation}
\mathcal{I}_{\Omega}(u) =(\boldsymbol{\nabla} u,  \boldsymbol{\nabla} v)_{\Omega} + \alpha( u,v)_{\Omega},\quad
\mathcal{N}_{\Gamma}(u) =  (u,v/\gamma -\boldsymbol{\nabla} v\cdot \boldsymbol{n} )_{\Gamma}.
\end{equation}

Now, we consider the generalized Nitsche's method, which -- after some manipulations -- can be written as  
\begin{equation}\label{eq:variational_R_cbm} 
\begin{split}
\mathcal{I}_{\Omega}(u) - (f,v)_{\Omega} -   (\boldsymbol{\nabla} u\cdot \boldsymbol{n},v )_{\Gamma}  + \dfrac{\gamma}{\gamma+\varepsilon}\mathcal{N}_{\Gamma}(u-u_{RD})
+ \dfrac{\varepsilon\gamma}{\gamma+\varepsilon}\mathcal{N}_{\Gamma}( \boldsymbol{\nabla} u\cdot \boldsymbol{n}-q_{RN}) = 0.
\end{split}
\end{equation}

We will now consider the asymptotic limits, $\varepsilon\rightarrow0$ and $\varepsilon^{-1} \rightarrow 0$, and the corresponding asymptotic developments in \eqref{eq:robin_D_expansion} and \eqref{eq:robin_N_expansion}. We can prove the following two results in the upcoming propositions.

\begin{proposition}[AP property of the Nitsche-Robin method in the Dirichlet limit $\varepsilon\rightarrow0$]
The asymptotic approximations obtained when $\varepsilon \rightarrow 0$ from the full Robin problem in \eqref{eq:variational_R_cbm} in the finite element space $\mathcal{V}^P_h$ for the asymptotic expansion modes $u_m$ in \eqref{eq:robin_D_expansion} define a cascade of variational problems which are all coercive and consistent with \eqref{eq:Robin_D_O0} through \eqref{eq:Robin_D_Om}. In particular, the $m$'th order can be obtained from the previous ones as the solution of
\begin{equation}\label{eq:asymptotics_cbm_D}
\mathcal{I}_{\Omega}(u_m) - ( f_m,v)_{\Omega} - (\boldsymbol{\nabla} u_m\cdot \boldsymbol{n},v )_{\Gamma} +\mathcal{N}_{\Gamma}(u_m-u_{D,m}) = \mathcal{E}_m,
\end{equation}
where the left-hand side defines the classical Nitsche-Dirichlet variational form with $f_0=f$ and $f_m=0$ for $m \ge 1$, while the right-hand side introduces asymptotic error terms with $\mathcal{E}_0=0$ and
\begin{equation}\label{eq:asymptotics_cbm_D_err}
\mathcal{E}_m = -\sum_{\ell = 0}^{m-1} (-\gamma)^{-(m-\ell)}\mathcal{N}_{\Gamma}(u_{\ell}-u_{D,\ell}), \quad \text{for} \quad m\ge 1.
\end{equation}
\end{proposition}

\begin{proof}
Equation \eqref{eq:asymptotics_cbm_D} and \eqref{eq:asymptotics_cbm_D_err} are readily obtained by introducing the ansatz from \eqref{eq:robin_D_expansion} in \eqref{eq:variational_R_cbm}, multiplying by $\gamma+\varepsilon$, and equating -- term by term -- all the coefficients of $\varepsilon^m$ powers. All problems are verified by the exact modes of \eqref{eq:Robin_D_O0}-\eqref{eq:Robin_D_Om}, which can be shown classically by integrating the domain equations by parts and using the boundary conditions. The coercivity of the classical Nitsche's variational form has been studied widely, e.g., see \cite{Benzaken2022} and references therein.
\end{proof}

Note that in the above development, the second Nitsche term in \eqref{eq:variational_R_cbm} involving the normal derivatives plays an important role when matching the correct Dirichlet values at all orders. This term affects the form of the Neumann asymptotic limit. In particular, we have the following equivalent result in the latter case.

\begin{proposition}[AP property of the Nitsche-Robin method in the Neumann limit $\varepsilon^{-1}\rightarrow 0$]
The asymptotic approximations obtained when $\varepsilon^{-1} \rightarrow 0$ from the full Robin problem in \eqref{eq:variational_R_cbm} in the finite element space $\mathcal{V}^P_h$ for the asymptotic expansion modes $u_m$ in \eqref{eq:robin_N_expansion} define a cascade of variational problems which are all coercive and consistent with \eqref{eq:Robin_N_O0} through \eqref{eq:Robin_N_Om}. In particular, the $m$'th order can be obtained from the previous ones as the solution of 
\begin{equation}\label{eq:asymptotics_cbm_N}
\mathcal{I}_{\Omega}(u_m) - (f_m,v)_{\Omega} - (q_{N,m} ,v)_{\Gamma}-  \gamma(\boldsymbol{\nabla} u_m\cdot \boldsymbol{n} -q_{N,m},\boldsymbol{\nabla} v\cdot \boldsymbol{n}) = \mathcal{E}_m,
\end{equation}
where the left-hand side defines a Neumann variational form with an additional penalty with $f_0=f$ and $f_m=0$ for $m\ge 1$, while the right-hand side introduces asymptotic error terms as $\mathcal{E}_0=0$ and
\begin{equation}\label{eq:asymptotics_cbm_N_err}
\mathcal{E}_m = \sum_{\ell = 0}^{m-1} (-\gamma)^{(m+1-\ell)}\mathcal{N}_{\Gamma}(\boldsymbol{\nabla} u_{\ell}\cdot \boldsymbol{n} -q_{N,\ell}),\quad\text{ for }\quad m\ge 1.
\end{equation}
\end{proposition}

\begin{proof}
Equation \eqref{eq:asymptotics_cbm_N} and \eqref{eq:asymptotics_cbm_N_err} are readily obtained by introducing  ansatz from \eqref{eq:robin_N_expansion} in \eqref{eq:variational_R_cbm}, multiplying by $\varepsilon^{-1} \gamma+1$, and equating -- term by term -- all the coefficients of $\varepsilon^{-m}$ powers. Integrating by parts of \eqref{eq:asymptotics_cbm_N} shows immediately the exact consistency with \eqref{eq:Robin_N_O0} to \eqref{eq:Robin_N_Om}, for all mode equations. Concerning coercivity, we make use of the trace inequality
\begin{equation}
\|h^{1/2}\boldsymbol{\nabla} u\|_{0,\Gamma} \le C_{\mathrm{I }}\|\boldsymbol{\nabla} u\|_{0,\Omega}.
\end{equation}
This allows us to argue that 
\begin{equation}
(\boldsymbol{\nabla} u_m,\boldsymbol{\nabla} u_m )_{\Omega} +\alpha( u_m, u_m )_{\Omega} -  \gamma(\boldsymbol{\nabla} u_m\cdot \boldsymbol{n},\boldsymbol{\nabla} u_m\cdot \boldsymbol{n})_{\Gamma}
\ge (1-h^{-1}C_{\mathrm{I}}\gamma)(\boldsymbol{\nabla} u_m,\boldsymbol{\nabla} u_m )_{\Omega} +\alpha( u_m, u_m )_{\Omega},
\end{equation}
which shows the coercivity of the variational form as soon as 
\begin{equation}
1-h^{-1}C_{\mathrm{I}}\gamma \ge C > 0.
\end{equation}
\end{proof}

\subsubsection{The Aubin-Robin method}
We proceed as in the previous section. With a similar notation, we recast the Aubin-Robin method as 
\begin{equation}\label{eq:variational_R_cbmA} 
\begin{split}
\mathcal{I}_{\Omega}(u) - (f,v)_{\Omega} -   (\boldsymbol{\nabla} u\cdot \boldsymbol{n},v )_{\Gamma}  + \dfrac{\gamma}{\gamma+\varepsilon}\mathcal{A}_{\Gamma}(u-u_{RD})
+ \dfrac{\varepsilon\gamma}{\gamma+\varepsilon}\mathcal{A}_{\Gamma}( \boldsymbol{\nabla} u\cdot \boldsymbol{n}-q_{RN}) = 0,
\end{split}
\end{equation}
where $\mathcal{I}_{\Omega}$ is defined as in \eqref{eq:notation}, while 
\begin{equation}\label{eq:notationA}
\mathcal{A}_{\Gamma}(u) =  (u,v/\gamma )_{\Gamma}.
\end{equation}
The scheme is formally identical to the Nitsche one, except for the definition of the penalty which does not include derivatives of the test function. Proceeding as before, we can prove the following two results. 

\begin{proposition}[AP property of the Aubin-Robin method in the Dirichlet limit $\varepsilon\rightarrow0$]
The asymptotic approximations obtained when $\varepsilon \rightarrow 0$ from the full Robin problem in \eqref{eq:variational_R_cbmA} in the finite element space $\mathcal{V}^P_h$ for the asymptotic expansion modes $u_m$ in \eqref{eq:robin_D_expansion} define a cascade of variational problems which are all coercive and consistent with \eqref{eq:Robin_D_O0} through \eqref{eq:Robin_D_Om}. In particular, the $m$'th order can be obtained from the previous ones as the solution of
\begin{equation}\label{eq:asymptotics_cbm_D_2}
\mathcal{I}_{\Omega}(u_m) - ( f_m,v)_{\Omega} - (\boldsymbol{\nabla} u_m\cdot \boldsymbol{n},v )_{\Gamma} +\mathcal{A}_{\Gamma}(u_m-u_{D,m}) = \mathcal{E}_m,
\end{equation}
where the left-hand side defines the classical Aubin-Dirichlet variational form with $f_0=f$ and $f_m=0$ for $m \ge 1$, while the right-hand side introduces asymptotic error terms with $\mathcal{E}_0=0$ and
\begin{equation}\label{eq:asymptotics_cbm_D_err_2}
\mathcal{E}_m = -\sum_{\ell = 0}^{m-1} (-\gamma)^{-(m-\ell)}\mathcal{A}_{\Gamma}(u_{\ell}-u_{D,\ell}), \quad \text{for} \quad m\ge 1.
\end{equation}
\end{proposition}

\begin{proof}
The proof of the consistency is identical to the one of Proposition 4.1 (not reported for brevity). The coercivity part of the proof is identical to the classical coercivity proof
for the Aubin-Dirichlet method.
\end{proof}

Similarly, in the Neumann asymptotic limit, we have the following equivalent result.

\begin{proposition}[AP property of the Aubin-Robin method in the Neumann limit $\varepsilon^{-1}\rightarrow 0$]
The asymptotic approximations obtained when $\varepsilon^{-1} \rightarrow 0$ from the full Robin problem in \eqref{eq:variational_R_cbmA} in the finite element space $\mathcal{V}^P_h$ for the asymptotic expansion modes $u_m$ in \eqref{eq:robin_N_expansion} define a cascade of variational problems which are all coercive and consistent with \eqref{eq:Robin_N_O0} through \eqref{eq:Robin_N_Om}. In particular, the $m$'th order can be obtained from the previous ones as the solution of 
\begin{equation}\label{eq:asymptotics_cbm_N_2}
\mathcal{I}_{\Omega}(u_m) - (f_m,v)_{\Omega} - (q_{N,m} ,v)_{\Gamma}  = \mathcal{E}_m,
\end{equation}
where the left-hand side defines a Neumann variational form with an additional penalty with $f_0=f$ and $f_m=0$ for $m\ge 1$, while the right-hand side introduces asymptotic error terms as $\mathcal{E}_0=0$ and
\begin{equation}\label{eq:asymptotics_cbm_N_err_2}
\mathcal{E}_m = \sum_{\ell = 0}^{m-1} (-\gamma)^{(m+1-\ell)}\mathcal{A}_{\Gamma}(\boldsymbol{\nabla} u_{\ell}\cdot \boldsymbol{n} -q_{N,\ell}),\quad\text{ for }\quad m\ge 1.
\end{equation}
\end{proposition}
\begin{proof}
The proof is identical to the one of Proposition 4.2 (not reported for brevity). The coercivity part of the proof is identical to the standard Neumanm method. 
\end{proof}

\subsubsection{On the stabilization parameter}
To conclude this section, we stress the impact of the stabilization parameter $\gamma$ on the stability and consistency order. In particular, the last proposition shows that for the Neumann limit to be coercive, the natural choice would be to set \begin{equation}\label{eq:gamma} 
\gamma = C_{\gamma} h,\quad C_{\gamma} C_{\mathrm{I}} \le 1 - C,
\end{equation}
for some strictly positive $C > 0$. This is somewhat in line with the fact that $\gamma^{-1}$ plays the role of the usual stabilization parameter in the Nitsche terms, $N_{\Gamma}$, and \eqref{eq:gamma} corresponds to the classical $h^{-1}$ scaling for this parameter, see \cite{Benzaken2022} and references therein.

As a side remark, the above choice -- which is standard for Nitsche's method -- affects the magnitude of the error terms in Proposition 4.1. In particular, note that \eqref{eq:asymptotics_cbm_D_err} involves scaling constants of the order $\gamma^{-m}$ which reduce the formal consistency with the higher order asymptotics by a factor of $h^{-m}$. This is of course only relevant when one is interested in reproducing higher ($\ge 2$) asymptotic modes.

\section{Embedded discrete equations via shifted boundary polynomial corrections}\label{sec:poisson_sbm}

In the following, we switch to a non-conformal mesh, whereupon we will solve \eqref{eq:Poisson} numerically. As discussed in Section \ref{sec:mappings}, we will thus work on a surrogate domain, $\overline{\Omega}$, with a surrogate boundary $\overline{\Gamma}$, split into Dirichlet, Neumann, and Robin boundaries, $\overline{\Gamma}_D$, $\overline{\Gamma}_N$, and $\overline{\Gamma}_R$, respectively (recall that the subscript $h$ is omitted for simplicity). The definition of the appropriate boundary conditions is obtained via the mapping in \eqref{eq:Gamma-map}. We thus define a weak form of the problem on the surrogate domain which can be abstractly written as
\begin{equation}
    a \langle u,v \rangle = b\langle v \rangle,  \addtocounter{equation}{1} \setcounter{lettercounter}{1}\tag{\theequation \alph{lettercounter}} \label{eq:weak_sbm_1}
\end{equation}
where the bi-linear and linear operators can be expanded, respectively, as
\begin{equation}
    a\langle u,v \rangle = a\langle u,v \rangle_{\overline{\Omega}} + a\langle u,v \rangle_{\overline{\Gamma}_D} + a \langle u,v \rangle_{\overline{\Gamma}_N} + a \langle u,v \rangle_{\overline{\Gamma}_R}, \addtocounter{equation}{0} \addtocounter{lettercounter}{1}\tag{\theequation \alph{lettercounter}} \label{eq:weak_sbm_bilinear} 
\end{equation}
\begin{equation}
    b\langle v \rangle = b\langle v \rangle_{\overline{\Omega}} + b\langle v \rangle_{\overline{\Gamma}_D} + b \langle v \rangle_{\overline{\Gamma}_N} + b\langle v \rangle_{\overline{\Gamma}_R},  \addtocounter{equation}{0} \addtocounter{lettercounter}{1}\tag{\theequation \alph{lettercounter}} \label{eq:weak_sbm_linear_1} 
\end{equation}
where 
\begin{equation}\label{eq:volume_var_terms}
a\langle u,v \rangle_{\overline{\Omega}} = (\boldsymbol{\nabla} u, \boldsymbol{\nabla} v)_{\overline{\Omega}}+ \alpha  ( u,  v)_{\overline{\Omega}},  \addtocounter{equation}{0} \addtocounter{lettercounter}{1}\tag{\theequation \alph{lettercounter}} 
\end{equation}
\begin{equation}\label{eq:volume_var_terms_2}
    b\langle v \rangle_{\overline{\Omega}} = (f, v)_{\overline{\Omega}}.  \addtocounter{equation}{0} \addtocounter{lettercounter}{1}\tag{\theequation \alph{lettercounter}} 
\end{equation}
To close the problem above, we need to specify boundary data on $\overline{\Gamma}$ as % and give explicit expressions for the linear and bilinear forms above,
\begin{align}
    \overline{u}_D    &= u(\overline{\boldsymbol{x}}),                                          & \text{for} \quad  \overline{\boldsymbol{x}} \in \overline{\Gamma}_D, \addtocounter{equation}{0} \addtocounter{lettercounter}{1}\tag{\theequation \alph{lettercounter}} \\
    \overline{q}_N    &= \boldsymbol{\nabla} u(\overline{\boldsymbol{x}}) \cdot \overline{\boldsymbol{n}},   &  \text{for} \quad  \overline{\boldsymbol{x}} \in \overline{\Gamma}_N, \addtocounter{equation}{0} \addtocounter{lettercounter}{1}\tag{\theequation \alph{lettercounter}} \\
    \overline{u}_{RD} &= u(\overline{\boldsymbol{x}}),                                          & \text{for} \quad  \overline{\boldsymbol{x}} \in \overline{\Gamma}_R, \addtocounter{equation}{0} \addtocounter{lettercounter}{1}\tag{\theequation \alph{lettercounter}} \\
    \overline{q}_{RN} &= \boldsymbol{\nabla} u(\overline{\boldsymbol{x}}) \cdot \overline{\boldsymbol{n}},   & \text{for} \quad  \overline{\boldsymbol{x}} \in \overline{\Gamma}_R. \addtocounter{equation}{0} \addtocounter{lettercounter}{1}\tag{\theequation \alph{lettercounter}}
\end{align}

We are now left with the challenge of how to define the surrogate boundary data above, given the data assigned on the true boundary $\Gamma$: $({u}_D, {q}_N, {u}_{RD}, {q}_{RN})$ in \eqref{eq:Poisson}. Using these values with no modification would result in a $\mathcal{O}(\| \boldsymbol{d} \|)$ error, thus limiting the method to first order.

\subsection{Shifted boundary polynomial corrections}

As discussed in the introduction, several approaches exist to correct this error between the two boundaries. In this work, we follow the shifted boundary approach introduced by Main \& Scovazzi for elliptic and parabolic problems \cite{Main2018a,Main2018b}, and later extended to various other applications \cite{Song2018,Li2020,Atallah2020a,Atallah2021a,Colomes2021,Li2021,Carlier2022}. In the original SBM formulation, a second-order accurate truncated Taylor expansion was introduced to match the data on boundaries. This approach has been shown to have a lot of potential in providing high-order approximations in embedded domains, avoiding classical drawbacks such as small cut-cell problems encountered in some cut-FEM methods or consistency limitations in immersed boundaries. The SBM only requires the knowledge of the mapping in \eqref{eq:Gamma-map} to enforce the given conditions while preserving the accuracy of the underlying finite element discretization. Some extensions beyond the second order have been proposed, initially in \cite{Nouveau2019}, and more recently in \cite{Ciallella2022,Atallah2022,Collins2023}. The strategy to correct the boundary data is the same in the high-order case: use a Taylor series development truncated at the proper order. As the accuracy increases, the number of derivative terms to be included in the series becomes quite large and cumbersome, especially in three spatial dimensions. A simplified formulation of the same method, avoiding the evaluation of all the Taylor basis terms at quadrature points, has been introduced in \cite{Ciallella2022}. We use the same one here and extend it to the case of Neumann and Robin boundary conditions. Despite being equivalent to the Taylor series-based SBM, the formulation proposed in \cite{Ciallella2022} involves a simple way to perform extrapolation using the underlying finite element approximation. In the following sections, we elaborate on the details of this technique.

 \subsubsection{Polynomial correction on the Dirichlet boundary}\label{sec:poly_cor_dir} % Polynomial correction on the Dirichlet boundary

Following the SBM ideas, we start by matching the Dirichlet data, $ u_D(\boldsymbol{x})$, with a $k$-th order multivariate Taylor series expansion of $u(\boldsymbol{\overline{x}})$ in a neighborhood of $\overline{\boldsymbol{x}} \in \overline{\Gamma}$. We consider at first the Nitsche's formulation. The adaptation of the same ideas to the penalty method is recalled later. In multinomial notation, this leads to
\begin{equation}\label{eq:SBM_D_first}
    u_D(\boldsymbol{x}) = u(\boldsymbol{\overline{x}}) + \mathcal{D}^k_D(\boldsymbol{d}) + \mathcal{O}(\| \boldsymbol{d}(\overline{\boldsymbol{x}}) \|^{k+1}), \quad \text{where} \quad \mathcal{D}^k_D = \sum_{\alpha = 1}^k \frac{1}{\alpha !}   \left ( \boldsymbol{\nabla}^{\alpha}  \right) \left ( u \right )(\boldsymbol{\overline{x}}) \boldsymbol{d}(\boldsymbol{\overline{x}})^{\alpha},
\end{equation}
where we recall that $\boldsymbol{d}=\boldsymbol{x}-\boldsymbol{\overline{x}}$ and $\mathcal{O}(\| \boldsymbol{d}(\overline{\boldsymbol{x}}) \|^{k+1})$ is the error associated with the truncated series. Neglecting this error, the above relation provides a high-order definition of the Dirichlet data on the surrogate boundary as
\begin{equation}\label{eq:polynomial_correction_dbc_1}
     \overline{u}_D = u(\boldsymbol{\overline{x}}) = u_D(\boldsymbol{x}) - \mathcal{D}^k_D.
\end{equation}
This is the original SBM idea studied in, e.g., \cite{Atallah2022,Collins2023} in the high-order case. In particular, as we will shortly see that \eqref{eq:polynomial_correction_dbc_1} can be inserted directly in the standard Nitsche formulation to obtain a modified variational form, which, on unfitted meshes, is $(k+1)$'th order accurate. That is, without the need to compute any mesh intersections or define any cut-cells. Thus avoiding many of the issues associated with these cells. Although not strictly necessary, symmetrizing terms can also be added to improve the formal stability of the method \cite{Main2018a,Nouveau2019,Collins2023}. In this work, we use a different view of the same method, which follows \cite{Ciallella2022}. The key idea introduced in the reference is to recognize that the expression
\begin{equation}
    p(\boldsymbol{x}) = u(\boldsymbol{\overline{x}}) + \mathcal{D}^k_D(\boldsymbol{x}-\boldsymbol{\overline{x}}),
\end{equation}
is simply the expansion of the polynomial $p$ onto the local Taylor basis $(1, x-\overline{x},y-\overline{y}, (x-\overline{x})^2/2, (y-\overline{y})^2/2, (x-\overline{x}) (y-\overline{y}), \text{etc.})$. The coefficients in this expansion are the derivatives of $u$, which are recovered from the underlying finite element polynomial in the standard SBM implementation, evaluated at $\boldsymbol{\overline{x}}$. The point-wise value obtained with this expansion is clearly the same as the one provided by any other basis. With this, there is no reason to change the basis at each quadrature point, and one can use all along the underlying finite element polynomial expansion used in the volume terms. In particular, for a polynomial of degree $k$, the Taylor term can be simply expressed as 
\begin{equation}
\mathcal{D}^k_D =    u(\boldsymbol{x}) - u(\boldsymbol{\overline{x}}). 
\end{equation}
This provides the following expression for the corrected value for the Dirichlet data on the surrogate boundary
\begin{equation}\label{eq:SPM_D_second}
    \overline{u}_D = u(\boldsymbol{\overline{x}}) +  u_D(\boldsymbol{x}) - u(\boldsymbol{x}).
\end{equation}
So, the SBM correction to match the boundary data can be written as the difference between the finite element cell polynomial evaluated on the true boundary and the data itself. The use of the available polynomial expansions avoids the cumbersome re-evaluations of all high-order derivatives. Following the original SBM idea, we now consider the standard Nitsche Dirichlet variational form with the modified data \eqref{eq:SPM_D_second}, which -- using the notation $\boldsymbol{x}=\mathcal{M}(\overline{\boldsymbol{x}})$ -- gives
\begin{equation}
\begin{split}
    a\langle u,v \rangle_{\overline{\Gamma}_D} &= \gamma^{-1} (u(\boldsymbol{x}),v)_{\overline{\Gamma}_D} - (\boldsymbol{\nabla} u \cdot \overline{\boldsymbol{n}},v)_{\overline{\Gamma}_D}  -(u(\boldsymbol{x}),\boldsymbol{\nabla} v \cdot \overline{\boldsymbol{n}})_{\overline{\Gamma}_D}, \\
    b\langle v \rangle_{\overline{\Gamma}_D} &= \gamma^{-1} ({u}_D,v)_{\overline{\Gamma}_D}-(u_D,\boldsymbol{\nabla} v \cdot \overline{\boldsymbol{n}})_{\overline{\Gamma}_D}.
\end{split}
\end{equation}
Note that if $u(\boldsymbol{x})$ is replaced by $u(\boldsymbol{\overline{x}}) + \mathcal{D}_D^k$ with $\mathcal{D}_D^k$ given by \eqref{eq:SBM_D_first}, we obtain the usual SBM formulation, e.g., see \cite{Atallah2022}. However, additional terms are added to enhance the symmetry of $a\langle u,v \rangle_{\overline{\Gamma}_D} $. In the classical SBM formulation, these terms involve the Taylor series terms applied to the test function. The symmetrized form can be written as 
\begin{equation}\label{eq:SBM_Ds}
\begin{split}
    a\langle u,v \rangle_{\overline{\Gamma}_D} &= \gamma^{-1} ( u(\boldsymbol{x}),v(\boldsymbol{x}))_{\overline{\Gamma}_D} - (\boldsymbol{\nabla} u \cdot \overline{\boldsymbol{n}},v)_{\overline{\Gamma}_D}  -(u(\boldsymbol{x}),\boldsymbol{\nabla} v \cdot \overline{\boldsymbol{n}})_{\overline{\Gamma}_D}, \\
    b\langle v \rangle_{\overline{\Gamma}_D} &= \gamma^{-1} ({u}_D,v(\boldsymbol{x}))_{\overline{\Gamma}_D}-(u_D,\boldsymbol{\nabla} v \cdot \overline{\boldsymbol{n}})_{\overline{\Gamma}_D}.
\end{split}
\end{equation}
Note that, for a full Dirichlet problem, the resulting method is exactly the same as the one of \cite{Atallah2022}, only its implementation is different as it involves a single polynomial kernel instead of a different Taylor basis and coefficients in each quadrature point on $\overline{\Gamma}_D$. Stability and error estimates can be found in \cite{Atallah2022}.

{\underline{\textit{For the Aubin method}}}: The use of the polynomial correction in the embedded Dirichlet terms leads to
\begin{equation}
\begin{split}
    a\langle u,v \rangle_{\overline{\Gamma}_D} &= \gamma^{-1} (u(\boldsymbol{x}),v)_{\overline{\Gamma}_D} - (\boldsymbol{\nabla} u \cdot \overline{\boldsymbol{n}},v)_{\overline{\Gamma}_D}, \\
    b\langle v \rangle_{\overline{\Gamma}_D} &= \gamma^{-1} ({u}_D,v)_{\overline{\Gamma}_D}.
\end{split}
\end{equation}

\subsubsection{Polynomial correction on the Neumann boundary}\label{sec:poly_cor_neu} % Polynomial correction on the Neumann boundary

Neumann boundaries are treated in a similar way. However, in this case, the boundary condition also depends on the normal, which must be accounted for in the correction. We consider the vectors
\begin{equation}
\overline{\boldsymbol{n}}=\overline{\boldsymbol{n}}(\boldsymbol{\overline{x}}), \quad 
\boldsymbol{n} = \boldsymbol{n}(\boldsymbol{x}) = \boldsymbol{n}(\boldsymbol{x}), \quad
 \boldsymbol{t} = \boldsymbol{t}(\boldsymbol{x}) = \boldsymbol{t}(\boldsymbol{x}),
\end{equation}
that represents the unit normal vector to the surrogate boundary and the corresponding perpendicular unit normal and tangent vectors to the true boundary. We then introduce the decomposition $\overline{\boldsymbol{n}}  = \boldsymbol{n}(\overline{\boldsymbol{n}} \cdot \boldsymbol{n}) + \boldsymbol{t}(\overline{\boldsymbol{n}} \cdot \boldsymbol{t})$
and write that
 \begin{equation}\label{eq:polynomial_correction_nbc_3}
    \overline{q}_N  = (\overline{\boldsymbol{n}} \cdot \boldsymbol{n}) \boldsymbol{\nabla} u(\overline{\boldsymbol{x}}) \cdot \boldsymbol{n}  + (\overline{\boldsymbol{n}} \cdot \boldsymbol{t})  \boldsymbol{\nabla} u(\overline{\boldsymbol{x}}) \cdot \boldsymbol{t}. 
\end{equation}
In the second term, we modify using the true Neumann data. As for the Dirichlet conditions, e.g., see \cite{Main2018a}, we introduce a Taylor series of the gradient in the fixed normal direction, $ \boldsymbol{n}$, which is matched with the true Neumann data as
\begin{equation}
    q_N(\boldsymbol{x}) = \boldsymbol{\nabla} u(\boldsymbol{\overline{x}}) \cdot \boldsymbol{n} + \mathcal{D}^k_N, \quad \text{where} \quad \mathcal{D}^k_N = \sum_{\alpha = 1}^k \frac{1}{\alpha !}   \left ( \boldsymbol{\nabla}^{\alpha}  \right) \left (\boldsymbol{\nabla} u \cdot \boldsymbol{n} \right )(\boldsymbol{\overline{x}}) \boldsymbol{d}(\boldsymbol{\overline{x}})^{\alpha}.
\end{equation}
As discussed in \cite{Nouveau2019}, one should be careful about how the derivative terms in the expansion above are obtained. If they are evaluated from a finite element cell polynomial of degree $k$, the $k$'th derivative is not consistent. Unless something special is done to recover this, the above development can only be performed to degree $k-1$. We use the identity 
\begin{equation}
 \mathcal{D}^k_N = \boldsymbol{\nabla} u(\boldsymbol{x}) \cdot \boldsymbol{n} - \boldsymbol{\nabla} u(\boldsymbol{\overline{x}}) \cdot \boldsymbol{n},
 \end{equation}
 and combine it with \eqref{eq:polynomial_correction_nbc_3}. After a few manipulations, we obtain
  \begin{equation}\label{eq:polynomial_correction_nbc_4}
    \overline{q}_N  
 = \boldsymbol{\nabla} u(\boldsymbol{\overline{x}}) \cdot \overline{\boldsymbol{n}} +  (\overline{\boldsymbol{n}} \cdot \boldsymbol{n}) (q_N(\boldsymbol{x}) - \boldsymbol{\nabla} u(\boldsymbol{x}) \cdot \boldsymbol{n}).
\end{equation}

Replacing this simple expression in the bilinear form and rearranging, we get
\begin{equation}\label{eq:SBM_Neumann}
\begin{split}
    a \langle u,v \rangle_{\overline{\Gamma}_N} &=  - ( \boldsymbol{\nabla} u(\boldsymbol{\overline{x}}) \cdot \overline{\boldsymbol{n}} ,v)_{\overline{\Gamma}_N} +   (\boldsymbol{\nabla} u(\boldsymbol{x}) \cdot \boldsymbol{n},(\overline{\boldsymbol{n}} \cdot \boldsymbol{n}) v)_{\overline{\Gamma}_N},\\
     b \langle v \rangle_{\overline{\Gamma}_N} &=  (q_N, (\overline{\boldsymbol{n}} \cdot \boldsymbol{n}) v)_{\overline{\Gamma}_N}.
\end{split}
\end{equation}

This modification corresponds to a correction term based on the difference between the Neumann data and the cell normal gradient of the finite element polynomial evaluated on the true boundary $\Gamma$. Note that in the current formulation, we obtain a sub-optimal method with consistency at most $\mathcal{O}(h^k)$ in the $L^2$ norm for an approximation of degree $k$. This is linked to the approximation accuracy of the first derivatives of the finite element polynomial.

\section{Asymptotic preserving Robin formulations: one implementation for all embedded boundaries}\label{sec:cons_robin} % Polynomial correction on the Robin boundary 

For the Robin conditions, we examine the idea of constructing AP approximations that can be used for a single high-order implementation for all types of boundary conditions. Thus, we focus on the full Robin problem in \eqref{eq:full_robin}, and start from the conformal weak formulation in Nitsche form given by \eqref{eq:variational_R_cbm}. With the notation of Section \ref{sec:AP_Robin}, we define the Dirichlet form and Neumann form (with added penalty as in Proposition 4.2) as
\begin{equation}
D := \mathcal{I}_{\Omega}(u) -(\partial_n u,v)_{\Gamma} -(f,v)_{\Omega}+\mathcal{N}_{\Gamma}(u-u_{RD}),
\end{equation}
\begin{equation}
N :=  \mathcal{I}_{\Omega}(u) -(\partial_n u,v)_{\Gamma} -(f,v)_{\Omega}+\mathcal{N}_{\Gamma}(\partial_n u-q_{RN}).
\end{equation}

As discussed in \cite{Juntunen2009} it is easy to see that \eqref{eq:variational_R_cbm} can be obtained as the linear combination
\begin{equation}
\dfrac{\gamma}{\gamma+\varepsilon}D + \dfrac{\varepsilon}{\gamma+\varepsilon} N =0.
\end{equation}
It is quite tempting to perform a similar combination for the embedded case to obtain a SBM Robin formulation. Unfortunately, this may lead to an inconsistent method. To see this, let us set
\begin{equation}\label{eq:overline_N}
\mathcal{N}_{\overline{\Gamma}}(u) := (u,v(\boldsymbol{x})/\gamma - \partial_{\bar{n}}v(\overline{\boldsymbol x})).
\end{equation}
Using \eqref{eq:SBM_Ds} and \eqref{eq:SBM_Neumann}, we define the embedded Robin problem from 
\begin{equation}
\overline{D} := \mathcal{I}_{\overline{\Omega}}(u) -(\partial_{\bar{n}} u(\overline{\boldsymbol x}),v)_{\overline{\Gamma}} -(f,v)_{\overline{\Omega}}+\mathcal{N}_{\overline{\Gamma}}(u(\boldsymbol{x})-u_{RD}),
\end{equation}
and  
\begin{equation}
\overline{N}:=  \mathcal{I}_{\overline{\Omega}}(u) -(\partial_{\bar{n}} u(\overline{\boldsymbol x}),v)_{\overline{\Gamma}} -(f,v)_{\overline{\Omega}}+
\gamma\mathcal{N}_{\overline{\Gamma}}((\overline{\boldsymbol{n}} \cdot \boldsymbol{n})(\partial_n u(\boldsymbol{x})-q_{RN})).
\end{equation}
Quick computations show that the method defined by
\begin{equation}
\dfrac{\gamma}{\gamma+\varepsilon}\overline{D} + \dfrac{\varepsilon}{\gamma+\varepsilon} \overline{N} = 0,
\end{equation}
leads to the variational formulation  
\begin{equation}\label{eq:SBM_Robin_nope}
 \mathcal{I}_{\overline{\Omega}}(u) -(\partial_{\bar{n}} u(\overline{\boldsymbol x}),v)_{\overline{\Gamma}} -(f,v)_{\overline{\Omega}} 
+\dfrac{\gamma}{\gamma+\varepsilon}  \mathcal{N}_{\overline{\Gamma}}(u(\boldsymbol{x})-u_{RD})\\
+\dfrac{\gamma\varepsilon }{\gamma+\varepsilon} \mathcal{N}_{\overline{\Gamma}}((\overline{\boldsymbol{n}} \cdot \boldsymbol{n})(\partial_n u(\boldsymbol{x})-q_{RN}))    =0.
\end{equation}
By construction, for $\varepsilon = 0$ this method reduces exactly to \eqref{eq:SBM_Ds}, and -- after some manipulations -- for $\varepsilon^{-1}\rightarrow 0$ to 
\begin{equation}
 \mathcal{I}_{\overline{\Omega}}(u) -(\partial_{\bar{n}} u(\overline{\boldsymbol x}),v)_{\overline{\Gamma}} -(f,v)_{\overline{\Omega}} 
 + (\partial_n u(\boldsymbol{x})-q_{RN},(\overline{\boldsymbol{n}} \cdot \boldsymbol{n})v)_{\overline{\Gamma}}
 - \gamma (\partial_n u(\boldsymbol{x})-q_{RN},(\overline{\boldsymbol{n}} \cdot \boldsymbol{n})\partial_{\bar n}v)_{\overline{\Gamma}} =0,
\end{equation}
which is equal to \eqref{eq:SBM_Neumann} with the extra penalty also found in \eqref{eq:asymptotics_cbm_N}. However, if we compute the Euler-Lagrange equations for general values of $\varepsilon$, we obtain
\begin{equation} 
\begin{split}
(-\boldsymbol{\nabla}^2 u +\alpha u - f,v)_{\overline{\Omega}}  + 
\dfrac{\gamma}{\varepsilon + \gamma} (\underbrace{u(\boldsymbol{x})-u_{RD} + \varepsilon (\overline{\boldsymbol{n}} \cdot \boldsymbol{n})
(\partial_n u(\boldsymbol{x})-q_{RN}) },v(\boldsymbol{x})/\gamma - \partial_{\bar{n}}v(\overline{\boldsymbol x}))_{\overline{\Gamma}} =0, 
\end{split}
\end{equation}
which in general is not consistent, even within a given power of $h$, due to the presence of the $(\overline{\boldsymbol{n}} \cdot \boldsymbol{n})$ product in the under-braced term that depends strongly on the generated mesh topology and true domain shape.

In the following sections, we propose three approaches that correct this problem. These methods will provide Robin formulations that are not only consistent but also enjoy an AP property. Hereby, they can be used to approximate any boundary condition simply by modifying the value of $\varepsilon$.

\subsection{The Nitsche-Robin method based on corrected linear combination coefficients} 

We first consider a simple correction of the coefficients in the linear combination discussed in the previous section. As with the inconsistent approach, we use \eqref{eq:SPM_D_second}   and \eqref{eq:polynomial_correction_nbc_4} to define $\overline{u}_{RD}$ and $\overline{q}_{RN}$, given $u_{RD}$ and $q_{RN}$. However, we now set $\bar{\gamma} =  \bar{n}_n  \gamma $, by introducing the short notation
\begin{equation}\label{eq:scalar_prod}
 \bar{n}_n := \overline{\boldsymbol{n}} \cdot \boldsymbol{n}.
\end{equation}
We now consider the modified variant of \eqref{eq:SBM_Robin_nope} that reads
\begin{equation}\label{eq:SBM_Robin_yes_a1}
 \mathcal{I}_{\overline{\Omega}}(u) -(\partial_{\bar{n}} u(\overline{\boldsymbol x}),v)_{\overline{\Gamma}} -(f,v)_{\overline{\Omega}} 
+  \mathcal{N}_{\overline{\Gamma}}\left ( \dfrac{\bar\gamma}{\bar\gamma+\varepsilon}(u(\boldsymbol{x})-u_{RD}) \right)\\
+ \mathcal{N}_{\overline{\Gamma}} \left (\dfrac{\bar\gamma\varepsilon }{\bar\gamma+\varepsilon}(\partial_n u(\boldsymbol{x})-q_{RN}) \right)    =0.
\end{equation}

The method is equivalent to the combination
\begin{equation}
\dfrac{\bar{\gamma}}{ \bar{\gamma}+\varepsilon}\overline{D} + \dfrac{\varepsilon}{ \bar{\gamma}+\varepsilon} \overline{N} = 0,
\end{equation}
which, however, due to the variability of $\bar{n}_n$, only makes sense when the combination coefficients are finally included in the boundary integrals as done in \eqref{eq:SBM_Robin_yes_a1}. A quick check shows that the limits for $\varepsilon \rightarrow 0$ and $\varepsilon^{-1} \rightarrow 0$ are the same as \eqref{eq:SBM_Robin_nope}. Moreover, the Euler-Lagrange equations in the general case read
\begin{equation}\label{eq:SBM_Robin_yes_a2}
\begin{split}
(-\boldsymbol{\nabla}^2 u +\alpha u - f,v)_{\overline{\Omega}}  +   
 (\underbrace{u(\boldsymbol{x})-u_D + \varepsilon  
(\partial_n u(\boldsymbol{x})-q_{RN}) },\dfrac{\bar\gamma}{\bar\gamma+\varepsilon}(v(\boldsymbol{x})/\gamma - \partial_{\bar{n}}v(\overline{\boldsymbol x})  ))_{\overline{\Gamma}} =0,
\end{split}
\end{equation}
where the under-braced term is consistent with the full Robin boundary condition within the accuracy of the extrapolation. For smooth enough problems, the latter is expected to behave as $\max(\mathcal{O}(h^{k+1}),\min(\varepsilon,C) \mathcal{O}(h^{k}))$ for some  $C>0$. More interestingly, the above approximation is asymptotic preserving, as shown below.

\begin{proposition}[AP property of the corrected coefficients Nitsche-Robin embedded method] The Robin embedded method with shifted boundary polynomial correction defined by \eqref{eq:SBM_Robin_yes_a1} is asymptotic preserving in both the Dirichlet limit $\varepsilon\rightarrow 0$ and in the Neumann limit $\varepsilon^{-1}\rightarrow 0$. In particular, we have that 
\begin{itemize}
\item in the limit $\varepsilon \ll 1$, the discrete modes $\{u_m\}_{m\ge 0}$ expanded as in \eqref{eq:robin_D_expansion} verify the coercive asymptotic variational problems
\begin{equation}\label{eq:AP-SBMa_D_v1}
 \mathcal{I}_{\overline{\Omega}}(u_m) -(\partial_{\bar{n}} u_m(\overline{\boldsymbol x}),v)_{\overline{\Gamma}} -(f_m,v)_{\overline{\Omega}} +  \mathcal{N}_{\overline{\Gamma}}( u_m(\boldsymbol{x})-u_{D,m})= \overline{\mathcal{E}}_m,
\end{equation}
with  $u_{D,m}$ defined by Proposition 2.1, and with $f_0=f$ and $f_m=0$ for $m\ge 1$, and where $\overline{\mathcal{E}}_0 =0$   while 
\begin{equation}\label{eq:AP-SBMa_D1}
\overline{\mathcal{E}}_m :=  -\sum_{\ell = 0}^{m-1} (-\bar\gamma)^{-(m-\ell)}\mathcal{N}_{\overline{\Gamma}}(u_{\ell}(\boldsymbol{x})-u_{D,\ell}),
\quad \forall m\ge 1.
\end{equation}

\item in the limit $\varepsilon^{-1} \ll 1$, discrete modes expanded as in  \eqref{eq:robin_N_expansion} verify the asymptotic variational problems
\begin{equation}\label{eq:AP-SBMa_N_v1}
 \mathcal{I}_{\overline{\Omega}}(u_m) -(\partial_{\bar{n}} u_m(\overline{\boldsymbol x}),v)_{\overline{\Gamma}} -(f_m,v)_{\overline{\Omega}} +  \mathcal{N}_{\overline{\Gamma}}(\bar \gamma ( \boldsymbol{\nabla} u_m(\boldsymbol{x})\cdot\boldsymbol{n} -q_{N,m})  )= \overline{\mathcal{E}}_m,
\end{equation}
with $q_{N,m}$ defined by Proposition 2.2, and with $f_0=f$ and $f_m=0$ for $m\ge 1$, and where $\overline{\mathcal{E}}_0 =0$ while 
\begin{equation}\label{eq:AP-SBMa_N1}
\overline{\mathcal{E}}_m :=     \sum_{\ell = 0}^{m-1} (-\bar\gamma)^{(m+1-\ell)}\mathcal{N}_{\overline{\Gamma}}(\boldsymbol{\nabla} u_{\ell}(\boldsymbol{x})\cdot \boldsymbol{n} -q_{N,\ell}), \quad \forall m\ge 1.
\end{equation}
\end{itemize}
\end{proposition}
\begin{proof}
To obtain the proof, we proceed by inserting the asymptotic developments formally into \eqref{eq:SBM_Robin_yes_a1}. To account for the effect of the $1/(\bar\gamma + \varepsilon)$ term, we cannot multiply the whole equation by it due to its variability in space. Instead, we replace this term by using a Maclaurin series expansion as
\begin{equation}
\dfrac{1}{1+x} =\sum_{n\ge 0}(-x)^n,
\end{equation}
applied to $x = \varepsilon \bar\gamma^{-1}$ in the limit $\varepsilon \ll 1$, and to $ \varepsilon^{-1} \bar\gamma$ in the limit $\varepsilon^{-1} \ll 1$.  Both \eqref{eq:AP-SBMa_D_v1} and \eqref{eq:AP-SBMa_N_v1} are obtained classically by gathering the terms multiplying equal powers of the small parameter. Each asymptotic problem is a Nitsche-type approximation of the appropriate order continuous problem, with error terms proportional to the boundary errors.  In particular, the variational forms \eqref{eq:AP-SBMa_D_v1} are equivalent to the symmetrized Nitsche SBM approximation of the corresponding continuous problems on non-conformal meshes, whose coercivity has been proven in \cite{Atallah2022}.
\end{proof}

As for the inconsistent approach in \eqref{eq:SBM_Robin_nope}, this method allows us to use the surrogate data independently in the Neumann and Dirichlet cases. Yet, the method restores the consistency with the Robin problem for general values of $\varepsilon$ and is asymptotic preserving. However, it involves coefficients depending on the scalar product in \eqref{eq:scalar_prod}, whose value cannot be a priori set. This modifies the weighting of the boundary terms in an uncontrolled -- albeit possibly small -- manner.

Concerning the parameter $\gamma$, the choice $\gamma =C_{\gamma} h$ allows us to ensure the coercivity of the bilinear form in the Neumann asymptotic limit for the conformal case. Unfortunately, stability results are only available for pure Dirichlet boundaries, and the impact of the penalty on this property seems to be relatively important \cite{Burman2012,Collins2023}. Thus, the same definition of $\gamma$ -- as the one given in the conformal case -- seems to be the most appropriate.

For completeness, we report the bi-linear and linear forms corresponding to the corrected scheme, which read
\begin{equation}\label{eq:NR-consistent1-var}
    \begin{split}
    a \langle u,v \rangle_{\overline{\Gamma}_R} &=  - ( \partial_{\bar n}u(\overline{\boldsymbol{x}}), v)_{\overline{\Gamma}_R} +
 \mathcal{N}_{\overline{\Gamma}_R} \left ( \dfrac{\bar\gamma}{\bar\gamma+\varepsilon}u(\boldsymbol{x}) \right)
+ \mathcal{N}_{\overline{\Gamma}_R} \left (\dfrac{\bar\gamma\varepsilon }{\bar\gamma+\varepsilon}\partial_n u(\boldsymbol{x}) \right ),\\
  b \langle v \rangle_{\overline{\Gamma}_R} &=
     \mathcal{N}_{\overline{\Gamma}_R} \left ( \dfrac{\bar\gamma}{\bar\gamma+\varepsilon} u_{RD} \right )+
 \mathcal{N}_{\overline{\Gamma}_R} \left (\dfrac{\bar\gamma\varepsilon }{\bar\gamma+\varepsilon}q_{RN} \right ).
\end{split}
\end{equation}

\subsection{The Nitsche-Robin method using polynomial corrections in the full Robin condition}

We now proceed differently by working directly with the full Robin condition. We modify the following SBM approach discussed in Sections \ref{sec:poly_cor_dir} and \ref{sec:poly_cor_neu} as
\begin{equation}
\begin{split}
u- u_{RD} +\varepsilon(\partial_n u - q_{RN}) = &u(\overline{\boldsymbol{x}}) + \mathcal{D}^k_D -u_{RD} +\mathcal{O}(\|\boldsymbol{d}^{k+1}\|)
+ \varepsilon(\partial_n u(\overline{\boldsymbol{x}})  +\mathcal{D}^k_N - q_{RN}) +\varepsilon\mathcal{O}(\|\boldsymbol{d}^k\|)\\=
&u(\overline{\boldsymbol{x}}) + \mathcal{D}^k_D -u_{RD} +\mathcal{O}(\|\boldsymbol{d}^{k+1}\|)
+ \varepsilon(\partial_{\bar n} u(\overline{\boldsymbol{x}})  +\mathcal{D}^k_N  + \mathcal{D}_N^n - q_{RN} ) +
\varepsilon\mathcal{O}(\|\boldsymbol{d}^k\|),
\end{split}
\end{equation}
having denoted by $\mathcal{D}^k_{D/N}$ the high-order polynomial corrections and set
\begin{equation}
 \mathcal{D}_N^n:= \partial_{ n} u(\overline{\boldsymbol{x}}) -\partial_{\bar n} u(\overline{\boldsymbol{x}}). 
\end{equation}
We now define
\begin{equation}\label{eq:Robin-data}
\overline{u}_{RD} := u_{RD} - \mathcal{D}^k_D,\quad \text{and} \quad
\overline{q}_{RN} := q_{RN}-\mathcal{D}^k_N  -\mathcal{D}_N^n,
\end{equation}
and apply the generalized Nitsche's method on the surrogate domain with this data. Using the polynomial corrections based on the cell finite element expansion, we obtain the problem
\begin{equation}\label{eq:SBM_Robin_yes_b1}
 \mathcal{I}_{\overline{\Omega}}(u) -(\partial_{\bar{n}} u(\overline{\boldsymbol x}),v)_{\overline{\Gamma}} -(f,v)_{\overline{\Omega}} 
+  \dfrac{\gamma}{\gamma+\varepsilon}\mathcal{N}_{\overline{\Gamma}}(u(\boldsymbol{x})-u_{RD})\\
+\dfrac{\gamma\varepsilon }{\gamma+\varepsilon} \mathcal{N}_{\overline{\Gamma}}(\partial_n u(\boldsymbol{x})-q_{RN}) = 0.
\end{equation}
A few computations show that the Euler-Lagrange equations for this problem are 
\begin{equation}\label{eq:SBM_Robin_yes_b2}
\begin{split}
(-\boldsymbol{\nabla}^2 u +\alpha u - f,v)_{\overline{\Omega}}  +   
\dfrac{\gamma}{\gamma+\varepsilon}
 (\underbrace{u(\boldsymbol{x})-u_D + \varepsilon  
(\partial_n u(\boldsymbol{x})-q_{RN}) },v(\boldsymbol{x})/\gamma - \partial_{\bar{n}}v(\overline{\boldsymbol x}  ))_{\overline{\Gamma}}=0,
\end{split}
\end{equation}
which readily shows the consistency with the Robin condition within an error of order $\max(\mathcal{O}(h^{k+1}),\min(\varepsilon,C) \mathcal{O}(h^{k}))$ for some constant $C>0$. The above approximation is also asymptotic preserving, and we can prove the following.

\begin{proposition}[AP property of the full Nitsche-Robin embedded method] The Robin embedded method with shifted boundary polynomial correction \eqref{eq:SBM_Robin_yes_b1} is asymptotic preserving in both the Dirichlet limit $\varepsilon\rightarrow 0$ and in the Neumann limit $\varepsilon^{-1} \rightarrow 0$. In particular, we have that 
\begin{itemize}
\item in the limit $\varepsilon \ll 1$, the discrete modes $\{u_m\}_{m\ge 0}$ expanded as in \eqref{eq:robin_D_expansion} verify the coercive asymptotic variational problems
\begin{equation}\label{eq:AP-SBMa_D_2}
 \mathcal{I}_{\overline{\Omega}}(u_m) -(\partial_{\bar{n}} u_m(\overline{\boldsymbol x}),v)_{\overline{\Gamma}} -(f_m,v)_{\overline{\Omega}} 
+  \mathcal{N}_{\overline{\Gamma}}( u_m(\boldsymbol{x})-u_{D,m})=
\overline{\mathcal{E}}_m,
\end{equation}
with $u_{D,m}$ defined by Proposition 2.1, and with $f_0=f$ and $f_m=0$ for $m\ge 1$, and where $\overline{\mathcal{E}}_0 =0$   while 
\begin{equation}\label{eq:AP-SBMa_D1_2}
\overline{\mathcal{E}}_m :=  -\sum_{\ell = 0}^{m-1} (-\gamma)^{-(m-\ell)}\mathcal{N}_{\overline{\Gamma}}(u_{\ell}(\boldsymbol{x})-u_{D,\ell}), \quad \forall m\ge 1.
\end{equation}
\item in the limit $\varepsilon^{-1} \ll 1$, discrete modes expanded as in \eqref{eq:robin_N_expansion} verify the asymptotic variational problems
\begin{equation}\label{eq:AP-SBMa_N}
 \mathcal{I}_{\overline{\Omega}}(u_m) -(\partial_{\bar{n}} u_m(\overline{\boldsymbol x}),v)_{\overline{\Gamma}} -(f_m,v)_{\overline{\Omega}} 
+  \mathcal{N}_{\overline{\Gamma}}( \gamma ( \boldsymbol{\nabla} u_m(\boldsymbol{x})\cdot\boldsymbol{n} -q_{N,m})  )=
\overline{\mathcal{E}}_m,
\end{equation}
with $q_{N,m}$ defined by Proposition 2.2, and with $f_0=f$ and $f_m=0$ for $m\ge 1$, and where $\overline{\mathcal{E}}_0 =0$ while 
\begin{equation}\label{eq:AP-SBMa_N1_2}
\overline{\mathcal{E}}_m :=     \sum_{\ell = 0}^{m-1} (-\gamma)^{(m+1-\ell)}\mathcal{N}_{\overline{\Gamma}}(\boldsymbol{\nabla} u_{\ell}(\boldsymbol{x})\cdot \boldsymbol{n} -q_{N,\ell}) \quad \forall m\ge 1.
\end{equation}
\end{itemize}
\end{proposition}
\begin{proof}
The proof follows the steps of the proof of Proposition 4.1 and 4.2. It is omitted for brevity.
\end{proof}

This approach has the advantage of not involving a parameter depending on the 
orientation of the surrogate boundary compared to the true boundary. For completeness, we report the bi-linear and linear forms corresponding to the corrected scheme, which read
\begin{equation}\label{eq:NR-consistent2-var}
    \begin{split}
    a \langle u,v \rangle_{\overline{\Gamma}_R} &=  - ( \partial_{\bar n}u(\overline{\boldsymbol{x}}), v)_{\overline{\Gamma}_R} +
 \dfrac{\gamma}{\gamma+\varepsilon} \mathcal{N}_{\overline{\Gamma}_R}( u(\boldsymbol{x}))
+\dfrac{\gamma\varepsilon }{\gamma+\varepsilon} \mathcal{N}_{\overline{\Gamma}_R}(\partial_n u(\boldsymbol{x})),\\
  b \langle v \rangle_{\overline{\Gamma}_R} &=\dfrac{\gamma}{\gamma+\varepsilon}
     \mathcal{N}_{\overline{\Gamma}_R}(  u_{RD})+\dfrac{\gamma\varepsilon }{\gamma+\varepsilon}
 \mathcal{N}_{\overline{\Gamma}_R}(q_{RN}).
\end{split}
\end{equation}

\subsection{The Aubin-Robin formulation}%MARIUZZ

We consider the embedded formulation of the Aubin-Robin obtained by simplifying the last approach. In particular, using the notation  
\begin{equation}\label{eq:overline_A}
\mathcal{A}_{\overline{\Gamma}}(u) := (u,v(\boldsymbol{x})/\gamma )_{\overline\Gamma},
\end{equation}
we consider the embedded Aubin-Robin method obtained setting as
\begin{equation}\label{eq:NR-consistent2-var_2}
    \begin{split}
    a \langle u,v \rangle_{\overline{\Gamma}_R} &=  - ( \partial_{\bar n}u(\overline{\boldsymbol{x}}), v)_{\overline{\Gamma}_R} +
 \dfrac{\gamma}{\gamma+\varepsilon}  \mathcal{A}_{\overline{\Gamma}_R}\!( u(\boldsymbol{x}))
+\dfrac{\gamma\varepsilon }{\gamma+\varepsilon} \mathcal{A}_{\overline{\Gamma}_R}\!(\partial_n u(\boldsymbol{x})),\\
  b \langle v \rangle_{\overline{\Gamma}_R} &=\dfrac{\gamma}{\gamma+\varepsilon}
     \mathcal{A}_{\overline{\Gamma}_R}\!(  u_{RD})+\dfrac{\gamma\varepsilon }{\gamma+\varepsilon}
 \mathcal{A}_{\overline{\Gamma}_R}\!(q_{RN}).
\end{split}
\end{equation}
This method can be shown to be AP, and in particular, enjoy the following property.

\begin{proposition}[AP property of the Aubin-Robin embedded method] The Robin embedded method with shifted boundary polynomial correction \eqref{eq:SBM_Robin_yes_b1} is asymptotic preserving in both the Dirichlet limit $\varepsilon\rightarrow 0$ and in the Neumann limit $\varepsilon^{-1} \rightarrow 0$. In particular, we have that 
\begin{itemize}
\item in the limit $\varepsilon \ll 1$, the discrete modes $\{u_m\}_{m\ge 0}$ expanded as in \eqref{eq:robin_D_expansion} verify the coercive asymptotic variational problems
\begin{equation}\label{eq:AP-SBMa_D}
 \mathcal{I}_{\overline{\Omega}}(u_m) -(\partial_{\bar{n}} u_m(\overline{\boldsymbol x}),v)_{\overline{\Gamma}} -(f_m,v)_{\overline{\Omega}} 
+  \mathcal{A}_{\overline{\Gamma}}( u_m(\boldsymbol{x})-u_{D,m})=
\overline{\mathcal{E}}_m,
\end{equation}
with $u_{D,m}$ defined by Proposition 2.1, and with $f_0=f$ and $f_m=0$ for $m\ge 1$, and where $\overline{\mathcal{E}}_0 =0$   while 
\begin{equation}\label{eq:AP-SBMa_D1_3}
\overline{\mathcal{E}}_m :=  -\sum_{\ell = 0}^{m-1} (-\gamma)^{-(m-\ell)}\mathcal{A}_{\overline{\Gamma}}(u_{\ell}(\boldsymbol{x})-u_{D,\ell}), \quad \forall m\ge 1.
\end{equation}
\item in the limit $\varepsilon^{-1} \ll 1$, discrete modes expanded as in \eqref{eq:robin_N_expansion} verify the asymptotic variational problems
\begin{equation}\label{eq:AP-SBMa_N_2}
 \mathcal{I}_{\overline{\Omega}}(u_m) -(\partial_{\bar{n}} u_m(\overline{\boldsymbol x}),v)_{\overline{\Gamma}} -(f_m,v)_{\overline{\Omega}} 
+  \mathcal{A}_{\overline{\Gamma}}( \gamma ( \boldsymbol{\nabla} u_m(\boldsymbol{x})\cdot\boldsymbol{n} -q_{N,m})  )=
\overline{\mathcal{E}}_m,
\end{equation}
with $q_{N,m}$ defined by Proposition 2.2, and with $f_0=f$ and $f_m=0$ for $m\ge 1$, and where $\overline{\mathcal{E}}_0 =0$ while 
\begin{equation}\label{eq:AP-SBMa_N1_3}
\overline{\mathcal{E}}_m :=     \sum_{\ell = 0}^{m-1} (-\gamma)^{(m+1-\ell)}\mathcal{A}_{\overline{\Gamma}}(\boldsymbol{\nabla} u_{\ell}(\boldsymbol{x})\cdot \boldsymbol{n} -q_{N,\ell}) \quad \forall m\ge 1.
\end{equation}
\end{itemize}
\end{proposition}
\begin{proof}
The proof follows the steps of the proof of Proposition 4.1 and 4.2. It is omitted for brevity.
\end{proof}

\begin{figure}[b]
    \centering
    \includegraphics{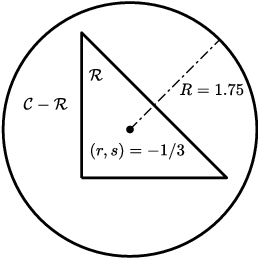}
    \caption{Graphical setup for the numerical domain when computing the Lebesque constants in interpolation on $\mathcal{R}$ and extrapolation on $\mathcal{C}-\mathcal{R}$.}
    \label{fig:Lebesque_constants}
\end{figure}

\subsection{Choice of the embedding: a note on interpolation error and conditioning}\label{sec:note_on_conditioning}

It is well known that the condition number of the system matrix, $\kappa(\mathcal{A})$, is of immense importance in relation to the cost of the matrix inversion and the accuracy of the resulting solution. Important factors for the value of $\kappa(\mathcal{A})$ are the choice of the finite element basis and associated quadrature points. The iso-parametric SEM is highly optimized for well-behaved interpolation, differentiation, and integration \textit{within} the element at the delicately picked quadrature points \cite{Warburton2006,Xu2018}. However, extrapolation may be as important as interpolation when considering unfitted finite elements. In the SBM setting using the classical extrapolation of $u(\boldsymbol{x})$ and $\boldsymbol{\nabla} u(\boldsymbol{x})$ will therefore result in evaluating the finite element polynomial outside the local element.

It is also well known that the Lebesque constant, $\Lambda_P$, is a good measure of the polynomial interpolation quality, e.g., see \cite{Pasquetti2010}. The Lebesque constant is defined by the Lebesque function, $\mathcal{L}_P = \mathcal{L}_P(\boldsymbol{r})$, in the reference domain, $\mathcal{R}$, as
\begin{equation}
    \Lambda_P = \underset{\boldsymbol{r} \in \mathcal{R}}{\max} ~\mathcal{L}_P(\boldsymbol{r}), \quad \text{where} \quad \mathcal{L}_P(\boldsymbol{r}) = \sum_{i = 1}^{N_{\text{ep}}} | h_i(\boldsymbol{r}) |.
\end{equation}

From \cite{Chen1995}, we see that the interpolation error in the maximum norm, $ \|f - \mathcal{I}_P f \|_{\infty}$, is bounded from above by the Lebesque constant as
\begin{equation}
    \|f - \mathcal{I}_P f \|_{\infty} \leq (1 + \Lambda_P) \| f - f^* \|_{\infty},
\end{equation}
where $f$ is some function with the interpolation $\mathcal{I}_P f$ over $\mathcal{R}$ and $f^* \in {\rm I\!P}^{P}$. The nodal Lagrange basis is bounded within the reference domain, yet, the basis functions grow unbounded outside at an exponential rate as higher and higher orders are considered. Hereby, it is expected that extrapolating the basis function will increase the Lebesque constant such that the interpolation error increases.

\begin{table}[t]
\centering
\caption{Lebesque constant, $\Lambda_P$, for polynomial orders, $P = \{1,...,10\}$, in interpolation and extrapolation as defined on Figure \ref{fig:Lebesque_constants} using the node distribution due to \cite{Warburton2006}.}
\begin{tabular}{r r r} \toprule
$P$  & Interpolation & Extrapolation \\ \midrule
1  & $1.00 \cdot 10^0$ & $2.76 \cdot 10^0$        \\ 
2  & $1.67 \cdot 10^0$ & $1.42 \cdot 10^1$       \\ 
3  & $2.11 \cdot 10^0$ & $8.03 \cdot 10^1$      \\ 
4  & $2.66 \cdot 10^0$ & $4.67 \cdot 10^2$     \\ 
5  & $3.12 \cdot 10^0$ & $2.73 \cdot 10^3$     \\ 
6  & $3.70 \cdot 10^0$ & $1.60 \cdot 10^4$    \\ 
7  & $4.27 \cdot 10^0$ & $9.44 \cdot 10^4$   \\ 
8  & $4.96 \cdot 10^0$ & $5.59 \cdot 10^5$  \\ 
9  & $5.74 \cdot 10^0$ & $3.32 \cdot 10^6$ \\ 
10 & $6.67 \cdot 10^0$ & $1.99 \cdot 10^7$ \\ \bottomrule
\end{tabular}
 \label{tab:Lebesque_constants}
\end{table}

To support this argument, a simple numerical exercise is considered inspired by \cite{Pasquetti2010}. The results are highlighted in Table \ref{tab:Lebesque_constants}, where the Lebesque constants are computed numerically for the classical interpolation case on $\mathcal{R}$ and similarly computed in extrapolation. The latter case is done for a circle of radius $R = 1.75$ with its center at the barycenter of the reference triangle, $(r,s) = -1/3$, i.e., $\mathcal{C} = \{ \boldsymbol{r} = (r,s) | (r + 1/3)^2 + (s + 1/3)^2 \leq 1.75^2 \}$. The extrapolation is done on $\mathcal{C}-\mathcal{R}$. This setup is visualized in Figure \ref{fig:Lebesque_constants}. From the Table, it is clear that the Lebesque constant -- and therefore also the extrapolation error --  grows rapidly and unbounded when increasing the polynomial orders. Note that the tabulated values for interpolation agree with the ones presented in \cite{Pasquetti2010,Hesthaven2007}.

 This motivates a comparison of the extrapolation-based SBM thoroughly with the one using the interpolation domain, as introduced in Section \ref{sec:mappings}. In particular, in the upcoming, we will compare the following four approaches denoted by the abbreviations: \textbf{CBM}, \textbf{SBM-e}, \textbf{SBM-ei}, and \textbf{SBM-i} as
\begin{enumerate}

    \item \textbf{CBM}: Body-fitted affine iso-parametric meshes where the elements conform to $\Gamma$. See Figure  \ref{fig:three_mappings} (top-left).
    
    \item \textbf{SBM-e}: Unfitted affine meshes, where $\overline{\Omega}$ is enclosed completely within $\Gamma$, and the mapping is defined in the original SBM fashion, i.e., along the true normal, $\boldsymbol{n}$, on $\Gamma$. This approach results in a pure extrapolation procedure as seen in Figure \ref{fig:three_mappings} (top-right). Note how all points are in pure extrapolation.
    
    \item \textbf{SBM-ei}: Unfitted meshes, where $\Gamma$ is within $\overline{\Omega}$, and the mapping is along $\boldsymbol{n}$ on $\Gamma$ as for the SBM-e. As visible in Figure \ref{fig:three_mappings} (lower-left), not all the quadrature points are mapped into the cells, thus resulting in a mixed extrapolation and interpolation procedure.
    
    \item \textbf{SBM-i}: Unfitted meshes, where $\Gamma$ is within $\overline{\Omega}$, and the mapping is defined locally such that any point $\overline{\boldsymbol{x}}$ of a surrogate boundary edge is mapped to a point $\boldsymbol{x} = \mathcal{M}(\overline{\boldsymbol{x}})$ within the element to which the edge belongs, as depicted in Figure \ref{fig:three_mappings} (lower-right). For this paper, we distribute the points on $\Gamma$ equidistantly between the edge intersections.

\end{enumerate}

\begin{figure}[h]
     \centering
     \begin{subfigure}
         \centering
         \includegraphics[width=0.35\textwidth]{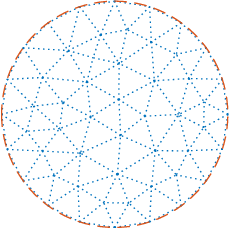}
     \end{subfigure}
     \hfill
     \begin{subfigure}
         \centering
         \includegraphics[width=0.35\textwidth]{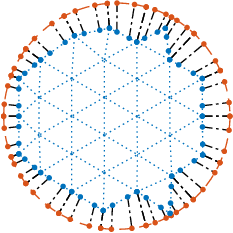}
     \end{subfigure}
     \hfill
     \begin{subfigure}
         \centering
         \includegraphics[width=0.35\textwidth]{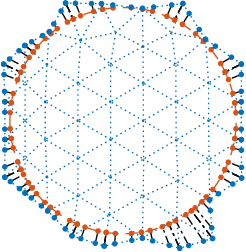}
     \end{subfigure}
     \hfill
     \begin{subfigure}
         \centering
         \includegraphics[width=0.35\textwidth]{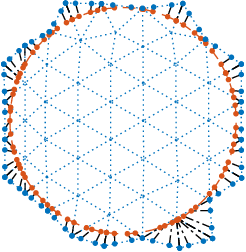}
     \end{subfigure}
        \caption{Illustrations of CBM (top-left) and the three SBM procedures, SBM-e (top-right), SBM-ei (lower-left), and SBM-i (lower-right), for $P=3$. The red line is $\Gamma$, the blue lines constitute the elements that form $\overline{\Omega}_h$, the red and blue dots represent nodal points $\boldsymbol{x} \in \Gamma$ and $\overline{\boldsymbol{x}} \in \overline{\Gamma}$, respectively, and the dashed black line is the mapping, $\mathcal{M}(\overline{\boldsymbol{x}})$.}
        \label{fig:three_mappings}
\end{figure}
 \section{Numerical Experiments}\label{sec:numerical_results}
In the following, we show a number of numerical experiments with Dirichlet, Neumann, and Robin boundary conditions. These will be of the Aubin and Nitsche variational forms using the four boundary and mapping approaches (CBM, SBM-e, SBM-ei, and SBM-i). See Figure \ref{fig:three_mappings} for the general concept hereof. We consider meshes composed of affine triangular elements that are quasi-uniform in size and arranged in an unstructured manner. Finite elements up to polynomial order $P=10$ are used on meshes that are generated with the open-source meshing software Gmsh \cite{Geuzaine2009}. An essential input to Gmsh is the characteristic length, $l_c$, which is an approximate size of the desired elements, $l_c \approx h_{\text{avg}}$. The true geometrical domain, $\Omega$, is of a circle with radius $R = 0.375$, unless stated otherwise. The circle is placed with center at $(x_0,y_0) = (L_x/2,L_y/2)$, with $L_x = 1$ and $L_y = 1$ (again, unless stated otherwise) being the domain lengths of the background mesh in the $x$ and $y$-directions, respectively. The proposed model is currently implemented in Matlab, version R2022a, and the linear system of equations in \eqref{eq:linear_system_of_equations} is solved with a direct solver within this framework.

For simplicity, we use the same exact solution throughout the paper obtained by the classical method of manufactured solutions (MMS). We construct the solution to the Poisson problem, such that $u_{\text{MMS}}(x,y) \in C^{\infty}$, by combining trigonometric functions with a linear polynomial twist as
\begin{equation}
    u_{\text{MMS}}(x,y) = \cos \left ( 5 \frac{x\pi}{L_x} \right )  \sin \left (5 \frac{y\pi}{L_y} \right ) + 2x - 1y.
\end{equation}
As an error measure, we exploit exact global integration through the SEM on the affine elements, with this $\|u_h - u \|_{L^1} = \int_{\Omega} |u_{\text{SEM}} - u_{\text{MMS}} | d \Omega$. Also, we opt to use a second measure as $\|\mathcal{A}u - b\|_{L^1}$ where $u = u_{\text{MMS}}$. This measure returns the level of truncation error without the influence of, e.g., conditioning issues from the matrix inversion.

\textbf{Outline of experiments}: We consider the three types of boundary conditions one by one. For Dirichlet, a verification of Aubin's and Nitsche's formulations is carried out on aligned meshes. This numerical experiment aims to -- qualitatively -- compare with other high-order SBM implementations in the literature. Then, we consider the solution on completely unstructured meshes, where $(\overline{\boldsymbol{n}} \cdot \ \boldsymbol{n})$ is not necessarily close to unity. For Neumann conditions, we limit ourselves -- for brevity -- to a single $p$-convergence study; however, general comments for the Neumann problem are included. For the Robin formulation, the consistent and inconsistent formulations are compared for Aubin and Nitsche. Also included is an assessment on unstructured meshes and high-order verification of the implementation. Lastly, we use the Robin condition with variable $\varepsilon$ to solve a mixed Dirichlet-Neumann problem and confirm the asymptotic preserving properties.

\subsection{Dirichlet boundary conditions}

First, we consider the pure Dirichlet problem, where we choose the penalty constant to be $\gamma^{-1} = 2/h_{\text{avg}}$, as this -- in the CBM case -- has provided the expected algebraic convergence rates in general. The authors note that the choice of an insufficient penalty constant can limit the convergence, yet too large values can amplify potential issues with ill-conditioning. A penalty constant that varies with the approximation order is presented in \cite{Atallah2022}, whereas a penalty-free high-order method is demonstrated in \cite{Collins2023}.

\subsubsection{Verification on aligned meshes}

A $h$-convergence comparison is carried out for $P = \{1,...,5\}$ for Aubin and Nitsche using CBM, SBM-e, SBM-ei, and SBM-i. In Appendix \ref{sec:appendix_DBC}, the figures are presented for Aubin in Figure \ref{fig:1_Aubin_DBC} and Nitsche in Figure \ref{fig:1_Nitsche_DBC}. We purposely align the surrogate boundary to the true boundary, such that $\overline{\boldsymbol{n}} \cdot \boldsymbol{n} \approx 1$. To do so, we generate the unfitted mesh by fitting another circle of radius $R \pm l_c/4$. The cases with "$+$" are for interpolation (SBM-ei and SBM-i), whereas "$-$" is for extrapolation (SBM-e). Meshes for CBM are constructed trivially. Alongside the $h$-convergence study, the associated matrix conditioning, $\kappa(\mathcal{A})$, is considered as well. 
As shown in all previous works, despite unfitted meshes, the shifted boundary approach genuinely provides high-order results with comparable errors. All results show optimal converged behavior of $\mathcal{O}(h_{\max}^{P+1})$ and align -- in magnitude and scaling -- with those presented in \cite{Atallah2022,Collins2023}. Ultimately, this confirms the potential of the SBM strategy. The condition number obtained is comparable in size where all the methods somewhat follow monotonically the $\mathcal{O}(1/h_{\max}^{2})$ scaling as for the CBM approach. This is one of the key features of the SBM approach due to the complete absence of small cut cells. Moreover, the Aubin formulation looks to provide slightly more regular and smooth rates compared to the Nitsche formulation. 

\subsubsection{Assessment on unstructured meshes}

\begin{figure}[t]
    \centering
    \includegraphics{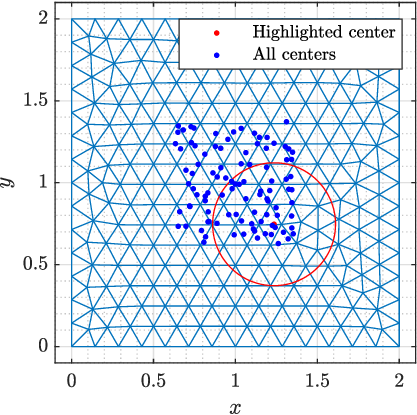}
    \caption{Mesh visualization for overall assessment. 100 circles of radius $R = 0.375$ with randomly generated centers. Highlighted center in red and all centers in blue.}
    \label{fig:Overall_assessment_mesh}
\end{figure}

When considering any of the shifted boundary formulations, we find some irregularities in the convergence properties occasionally. Those irregularities are due to mesh sensitivity, i.e., where and how the true domain is embedded in the background mesh. Here, the distance between the true and surrogate boundaries can be small or larger; the boundaries can be almost parallel or almost perpendicular, etc. Due to this, the convergence rates are not always perfectly monotonic, as seen multiple times in SBM literature \cite{Nouveau2019,Collins2023}. Therefore, we wish to compare SBM-e, SBM-ei, and SBM-i more quantitatively to justify numerically whether the interpolating approaches improve the numerical scheme over the classical extrapolation approach. For this, a $(L_x \times L_y)$ square domain is considered with $L_x = L_y = 2$ which is is meshed using $l_c = 0.15$. Now, 100 circles are embedded into the background mesh (one at a time) given a randomly generated center as shown in Figure \ref{fig:Overall_assessment_mesh}.

For each of the 100 mesh configurations, we consider $p$-convergence studies and the associated matrix conditioning. The median and the variance of $\|\mathcal{A}u - b \|_{L^1}$ and $\kappa(\mathcal{A})$ of the 100 discrete solutions are then computed and displayed in \ref{fig:3_DBC} for SBM-e, SBM-ie, and SBM-e with both Aubin and Nitsche formulations. Regarding the median, we observe for the Aubin formulation how SBM-ei and SBM-i provide a more favorable convergence rate than SBM-e. A similar conclusion is present for the Nitsche formulations; however, SBM-i looks to be slightly advantageous over SBM-ei. The same observation is seen in the conditioning median, especially for $P > 3$. The differences between interpolation and extrapolation are comparable to the tabulated values in Table \ref{tab:Lebesque_constants}. Also, it can be noted that as long as SBM-i is used, there is a minor difference between Aubin and Nitsche in convergence and conditioning. Regarding the variances, SBM-e generally provides the largest compared to the two interpolation approaches. This is true for both convergence and conditioning. Ultimately, the results emphasize how both the Aubin and Nitsche formulations favor using interpolation over extrapolation of the basis function. The authors note that the presented numerical analysis is based on one single mesh; thus, it is not general proof. However, we have not seen any contradictory results in the research for this paper.

\begin{figure}[t]
    \centering
    \includegraphics{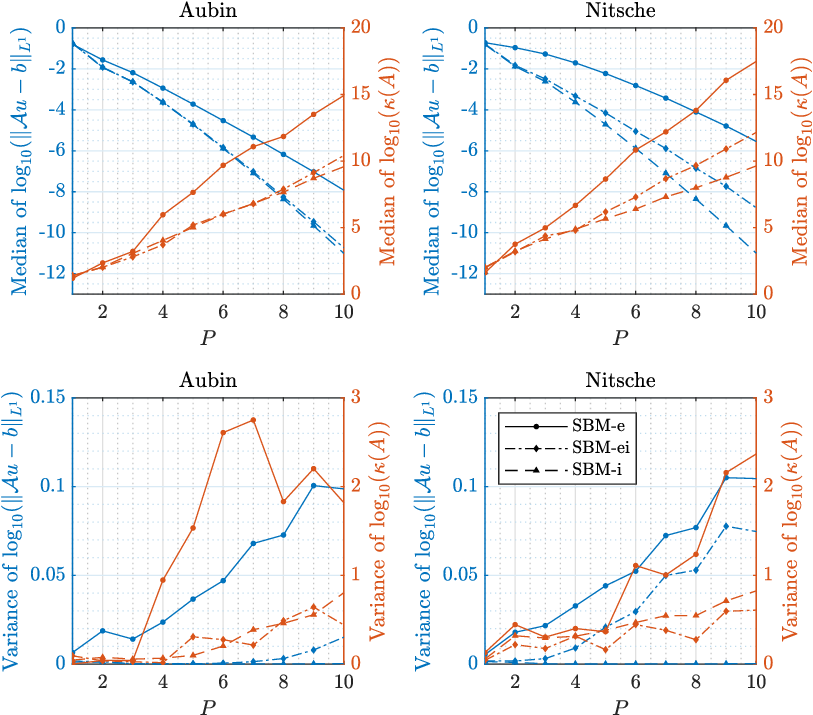}
    \caption{Median (top row) and variance (bottom row) of 100 discrete solutions with the setup as in Figure \ref{fig:Overall_assessment_mesh} for the Dirichlet problem. Quantities are for $\|\mathcal{A}u - b \|_{L^1}$ and $\kappa(\mathcal{A})$ for Aubin and Nitsche formulation with SBM-e, SBM-ei, and SBM-i.}
    \label{fig:3_DBC}
\end{figure}

\subsection{Neumann boundary conditions} 

Next, the pure Neumann problem is considered. For conciseness, we only present one numerical study, namely $p$-convergence on an aligned mesh with $l_c = 1/16$ using SBM-e and SBM-i on two boundary formulations: 
\begin{itemize}
    \item Formulation 1: \textit{Standard form.}
    \begin{equation*}
        a \langle u,v \rangle_{\overline{\Gamma}_N} = 0, \quad \text{and} \quad b \langle v \rangle_{\overline{\Gamma}_N} = (\overline{q}_N,v)_{\overline{\Gamma}_N}
    \end{equation*}
    
     \item Formulation 2: \textit{Standard form with symmetrical term.}
     \begin{equation*}
          a \langle u,v \rangle_{\overline{\Gamma}_N} = - \gamma^{-1} (\boldsymbol{\nabla} u \cdot \boldsymbol{n},\boldsymbol{\nabla} v \cdot \boldsymbol{n})_{\overline{\Gamma}_N}, \quad \text{and} \quad b \langle v \rangle_{\overline{\Gamma}_N} = (\overline{q}_N,v)_{\overline{\Gamma}_N} - \gamma^{-1} (\overline{q}_{N},\boldsymbol{\nabla} v \cdot \boldsymbol{n})_{\overline{\Gamma}_N}.
     \end{equation*}
\end{itemize}
The reasoning behind considering these two formulations is to highlight the behavior of the upcoming Robin Aubin form (Formulation 1) and Robin Nitsche form (Formulation 2) in the Neumann limit as $\varepsilon^{-1} \xrightarrow{} 0$. The results are shown in Figure \ref{fig:4_NBC}, where it can be observed that all combinations yield spectral convergence on the mesh. It can be noted that the convergence rate is ruined for Formulation 2 using SBM-e, which is due to ill-conditioning. Moreover, a decrease in convergence rate is seen when using SBM-e over SBM-i. Also, the condition number is increased when: i) using Formulation 2 over Formulation 1 and ii) using SBM-e over SBM-i. 

For conciseness, further results for the Neumann problem have been intentionally left out. We stress that a sub-optimal $\mathcal{O}(h_{\max}^{P})$ convergence rate is obtained for all SBM-e, SBM-ei, and SBM-i approaches. For CBM, optimal rate of $\mathcal{O}(h_{\max}^{P+1})$ is reached. This sub-optimality is expected -- as already discussed in Section \ref{sec:poly_cor_neu} -- due to the polynomial correction. Some embedded formulations that allow to recover the optimal rates are proposed, e.g., in \cite{Lozinski2019,Nouveau2019,Duprez2023}. Also, a general increase in conditioning is seen compared to the Dirichlet problem. The explanation hereof is due to gradients in the basis functions as required in the Neumann formulations. The SEM basis functions are defined inside the elements and grow rapidly and unbounded outside the element; hence, the gradients will enhance this polynomial behavior even more outside the element. Ultimately, this will imply an increase in the conditioning. In general, SBM-e provides larger $\kappa(\mathcal{A})$ (multiple orders of magnitude) than SBM-ei and SBM-i. 

\begin{figure}[t]
    \centering
    \includegraphics{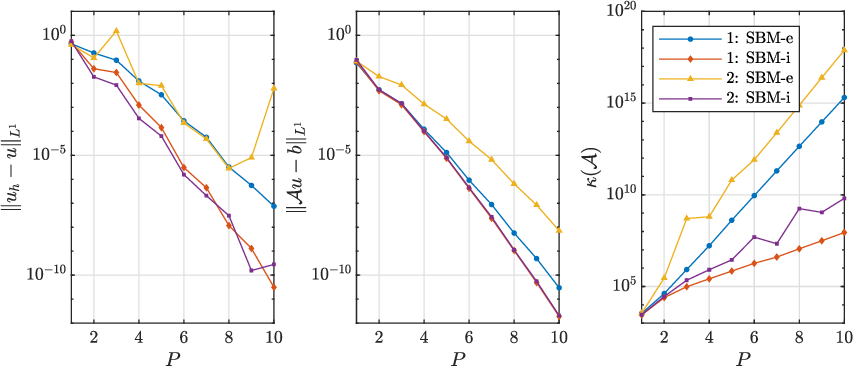}
    \caption{$p$-convergence and associated condition number study of the Neumann problem on a mesh with $l_c = 1/16$. The analysis is of two boundary formulations, Formulation 1 and Formulation 2, for SBM-e and SBM-i.}
    \label{fig:4_NBC}
\end{figure}

\subsection{Robin boundary conditions}

The last problem to consider is the pure Robin formulation. For this, we set $\gamma^{-1}$ as with the Dirichlet problem. 

\subsubsection{Consistent versus inconsistent formulations}

As outlined in Section \ref{sec:cons_robin}, the \textit{off-the-selves} Aubin and Nitsche formulations for embedded Robin boundary conditions are inconsistent if the standard polynomial corrections are applied. Therefore, we proposed a consistent way of formulating both variational statements. Ultimately, we are left with the following inconsistent and consistent equations on the boundary derived from the Euler-Lagrange equations as
\begin{align*}
    \text{Inconsistent:}& \quad u(\boldsymbol{x})-u_{RD} + \varepsilon (\overline{\boldsymbol{n}} \cdot \boldsymbol{n})
(\nabla u(\boldsymbol{x}) \cdot \boldsymbol{n} - q_{RN}) = 0. \\
    \text{Consistent:}& \quad u(\boldsymbol{x})-u_{RD} + ~~~~~~~~~~~~ \varepsilon 
(\nabla u(\boldsymbol{x}) \cdot \boldsymbol{n} - q_{RN}) = 0.
\end{align*}
Recall that $\boldsymbol{x} = \mathcal{M}(\overline{\boldsymbol x})$. Now, $u_{RD}$ and $q_{RN}$ are constructed from the analytical MMS solution. Therefore, we will not be able to locate the consistency issue as the error is reduced for $(\nabla u(\boldsymbol{x}) \cdot \boldsymbol{n} - q_{RN})$. To circumvent this, we perturb the Robin boundary condition for \eqref{eq:full_robin} with an arbitrary scalar value, $\delta$, as
\begin{equation}
    u + \varepsilon ( \boldsymbol{\nabla} u \cdot \boldsymbol{n})  =  u_{RD} - \delta + \varepsilon (q_{RN} + \delta/\varepsilon) ,  \quad \text{on} \quad \Gamma_R.
\end{equation}
With this, the consistent and inconsistent boundary equations give
\begin{align*}
    \text{Inconsistent:}& \quad u(\boldsymbol{x})-u_{RD} + \varepsilon (\overline{\boldsymbol{n}} \cdot \boldsymbol{n})
(\nabla u(\boldsymbol{x}) \cdot \boldsymbol{n} - q_{RN})  =  \delta( \overline{\boldsymbol{n}} \cdot \boldsymbol{n} - 1). \\
    \text{Consistent:}& \quad u(\boldsymbol{x})-u_{RD} + ~~~~~~~~~~~~ \varepsilon 
(\nabla u(\boldsymbol{x}) \cdot \boldsymbol{n} - q_{RN})  = \delta - \delta = 0.
\end{align*}
From this, we see that the inconsistency error -- in a weak sense -- becomes $\delta(\overline{\boldsymbol{n}} \cdot \boldsymbol{n} - 1)$, and that the consistent formulation still satisfies the $\delta$ perturbation as expected.

In Figure \ref{fig:5_RBC_consistency}, a $p$-convergence study is carried for $P = \{1,...,10\}$ on four different meshes, $l_c = \{1/4, 1/8, 1/16, 1/32 \}$. Here, both the consistent and inconsistent formulations for Aubin and Nitsche are used with $\varepsilon = \delta = 1$. Ultimately, the results show how the truncation error stalls at certain levels for both of the inconsistent formulations. These levels are of first order in relation to the mesh size, i.e., related to $(\overline{\boldsymbol{n}} \cdot \boldsymbol{n})$ that approach unity for $l_c \xrightarrow{} 0$, as expected. Both of the consistent formulations can be seen to converge properly with exponential rate to machine precision. Now that the inconsistency is located we opt to solely use the consistent formulations for the remaining part of the paper.

\begin{figure}[t]
    \centering
    \includegraphics{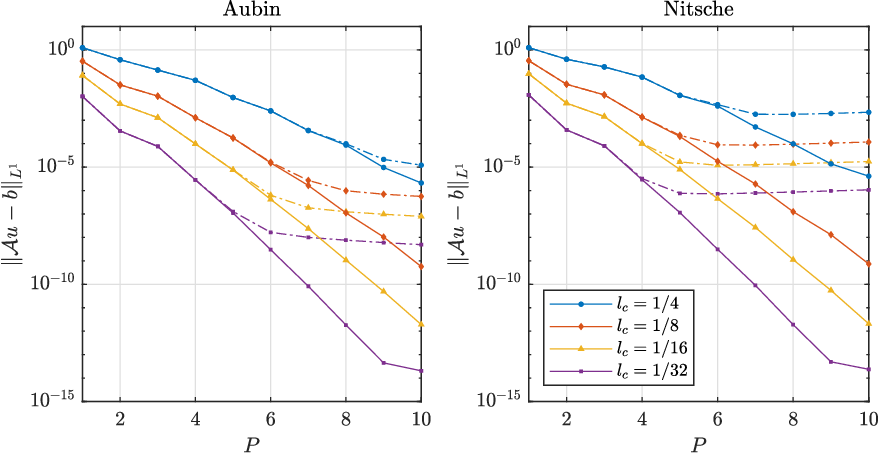}
    \caption{Consistent formulation (full lines) versus inconsistent formulation (dotted-dashed lines). For both Aubin (left) and Nitsche (right) on four different meshes, $l_c = \{1/4, 1/8, 1/16, 1/32 \}$, with $p$-convergence on each. Results using SBM-i on aligned meshes with $\varepsilon = \delta = 1$.}
    \label{fig:5_RBC_consistency}
\end{figure}

\subsubsection{Assessment and verification on unstructured meshes}

We performed the same statistical assessment as for the pure Dirichlet problem (recall Figure \ref{fig:Overall_assessment_mesh}). For this, we use $\varepsilon = 1$, and the consistent Robin formulations. The analysis can be seen in Figure \ref{fig:7_RBC} in Appendix \ref{sec:appendix_RBC}, ultimately confirming the observations and conclusions from the Dirichlet case.

Next, a verification of the Robin implementation -- in terms of $h$-convergence studies -- is to be carried out. For this, we select three different values of $\varepsilon = \{10^{-1},10^{0},10^{1} \}$ and run the analysis for $P = \{2,6,10 \}$ for both Aubin and Nitsche. We compare the SBM-i approach to the reference CBM solution. Recall that the latter converges optimally in both the Dirichlet and Neumann limits, whereas the polynomial corrected SBM methods converge sub-optimal in the Neumann limit and optimal in the Dirichlet limit. The results can be seen in Figure \ref{fig:8_RBC}. From the figure, it can be observed how the SBM-i genuinely agrees with the CBM rates, especially in the low $\varepsilon$ case. Moreover, as $\varepsilon$ increases, the difference in convergence rates (optimal versus sub-optimal) appears. Lastly, the authors note the slight irregularities in the SBM-i results, which is due to the aforementioned mesh sensitivity of the SBM implementations.

\begin{figure}[t]
    \centering
    \includegraphics{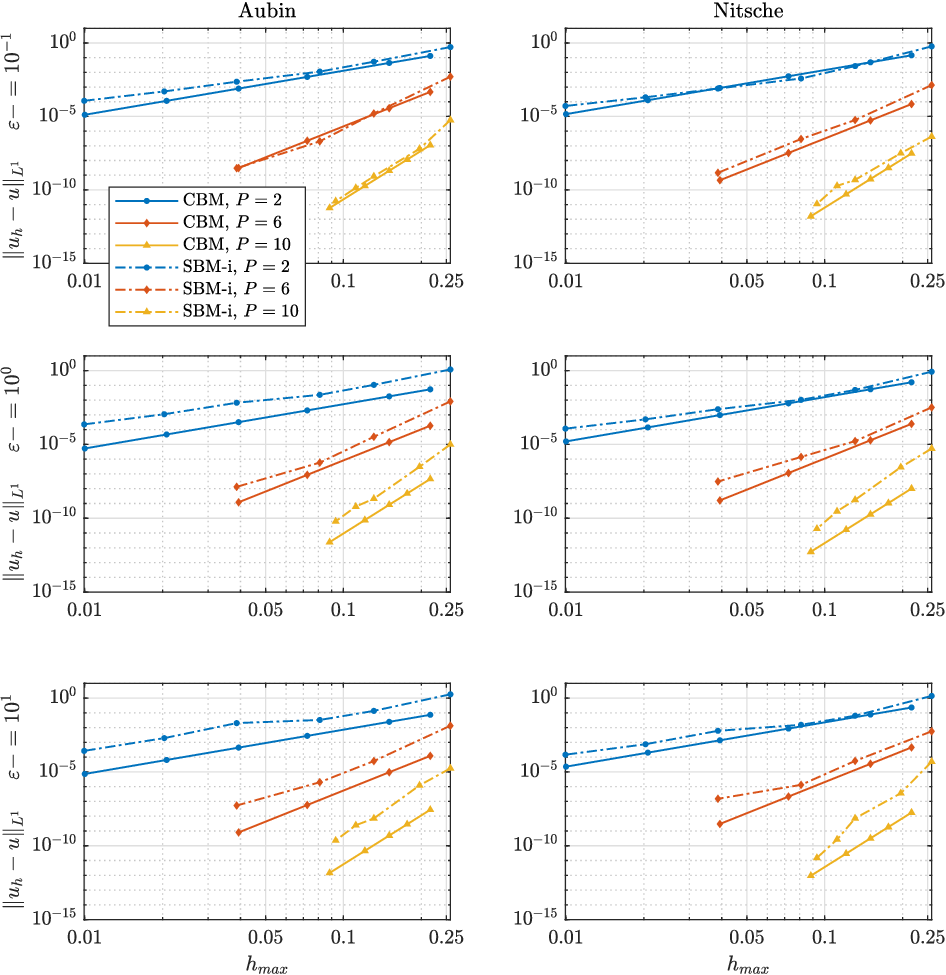}
    \caption{$h$-convergence study for the consistent Robin formulations (Aubin and Nitsche) with different $\varepsilon = \{10^{-1},10^{0},10^{1} \}$ and $P = \{2,6,10 \}$. The analysis is with CBM and SBM-i.}
    \label{fig:8_RBC}
\end{figure}

\newpage

\subsubsection{A mixed Dirichlet-Neumann problem using only Robin boundary conditions}

To demonstrate the implemented \textit{all-in-one} strategy, we perform a $p$-convergence study on the mesh layout given in Figure \ref{fig:Mesh_square_with_circle}. Here, a problem with outer Dirichlet and inner Neumann conditions is considered by setting $\varepsilon$ properly in the Robin formulation. In this case, $\varepsilon = 10^{-10}$ on the outer square and $\varepsilon = 10^{10}$ on the inner circle.

\begin{figure}[t]
    \centering
    \includegraphics{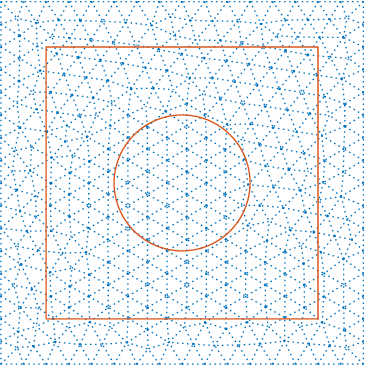}
    \caption{Mesh layout with elements of size $l_c = 1/8$ of a $(1.5 \times 1.5)$ square domain with a circular hole of radius $R = 0.375$ at the center $(x_0,y_0) = (1,1)$ is considered. Background mesh domain is of lengths $L_x = L_y = 2$. Robin boundary conditions with $\varepsilon = 10^{-10}$ (Dirichlet mimic) on square and $\varepsilon = 10^{10}$ (Neumann mimic) on circle.}
    \label{fig:Mesh_square_with_circle}
\end{figure}

In Figure \ref{fig:9_RBC_square_with_circle}, we perform the usual $p$-convergence analysis on three different meshes. Here, spectral convergent behavior for all computations is observed. Moreover, q ruined convergence rate for very high orders using SBM-e -- due to ill-conditioning -- is seen as well. The SBM-i provides similar-sized errors independent of choosing Aubin or Nitsche. 

\begin{figure}[b]
    \centering
    \includegraphics{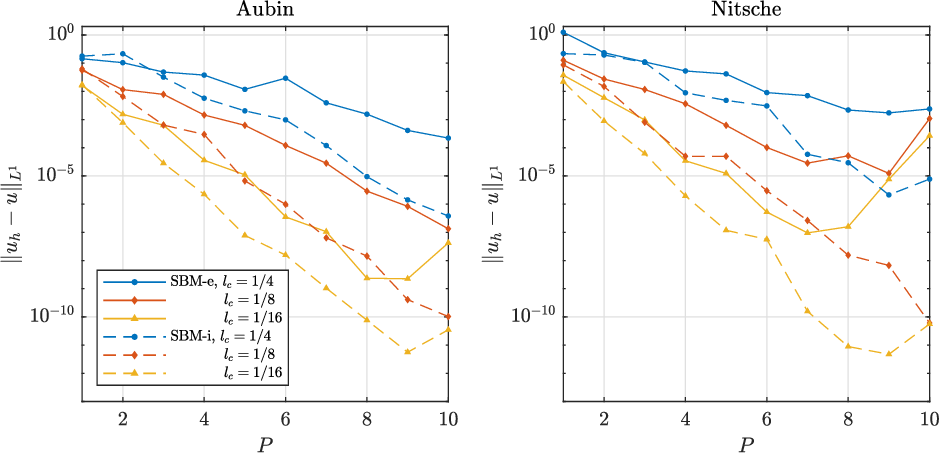}
    \caption{$p$-convergence study on mesh layout presented in Figure \ref{fig:Mesh_square_with_circle} with $l_c = \{1/4, 1/8, 1/16 \}$. For both Aubin and Nitsche with SBM-e and SBM-i.}
    \label{fig:9_RBC_square_with_circle}
\end{figure}

\subsubsection{Confirmation of asymptotic preserving properties}

As a last numerical experiment, we seek to confirm the asymptotic preserving properties of the implementation. Recall the solution perturbations in Proposition 2.1 and Proposition 2.2 for the full Robin solution, $u$. For completeness, in the Dirichlet limit, we have the following expansion (only including up to second-order terms) as
\begin{equation*}
    u = u_0 + \varepsilon u_1 + \varepsilon^2 u_2 + \mathcal{O}(\varepsilon^3).
\end{equation*}

In this perturbed setting, we solve the full Robin problem for $u$ (for some arbitrary $\varepsilon$) and the zeroth order Dirichlet problem for $u_0$ on a mesh given some approximation order, $P$. Then, $u_1$ and $u_2$ are solved recursively as given in Proposition 2.1. From these solutions, we compute $\|u-u_0 \|_{L^1}$, $\|u-u_0-\varepsilon u_1 \|_{L^1}$, and $\|u-u_0 - \varepsilon u_1 - \varepsilon^2 u_2 \|_{L^1}$ for different values of $\varepsilon$. The results can be seen in Figure \ref{fig:6_Aubin_RBC_AP}, for $P = \{1,...,10 \}$ on a mesh with $l_c = 1/8$. The convergence studies are supported by contour plots of $u-u_0$ versus $\varepsilon u_1$ and $u - u_0 - \varepsilon u_1$ versus $\varepsilon^2 u_2$ for different values of $\varepsilon$: $\varepsilon = 10^{-2}$ in Figure \ref{fig:6_Aubin_RBC_AP_Contour_P_10_eps_1e-2}, $\varepsilon = 10^{-4}$ in Figure \ref{fig:6_Aubin_RBC_AP_Contour_P_10_eps_1e-4}, and $\varepsilon = 10^{-6}$ in Figure \ref{fig:6_Aubin_RBC_AP_Contour_P_10_eps_1e-6}. For all contour plots, $P = 10$, where the solution is interpolated onto the true circle for the purpose of visualization.

\begin{figure}[H]
    \centering
    \includegraphics{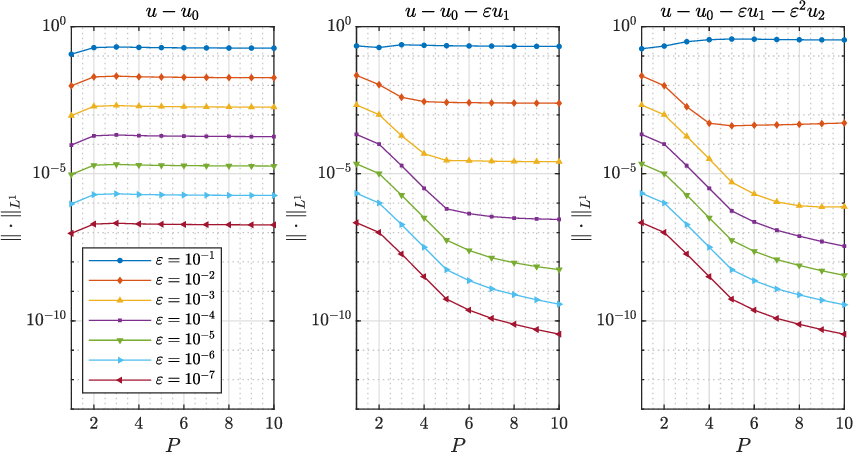}
    \caption{$p$-convergence on a mesh with $l_c = 1/8$ and variable $\varepsilon$ in $\|u-u_0 \|_{L^1}$, $\|u-u_0-\varepsilon u_1 \|_{L^1}$, and $\|u-u_0 - \varepsilon u_1 - \varepsilon^2 u_2 \|_{L^1}$.}
    \label{fig:6_Aubin_RBC_AP}
\end{figure}

\begin{figure}[H]
    \centering
    \includegraphics[scale = 0.6]{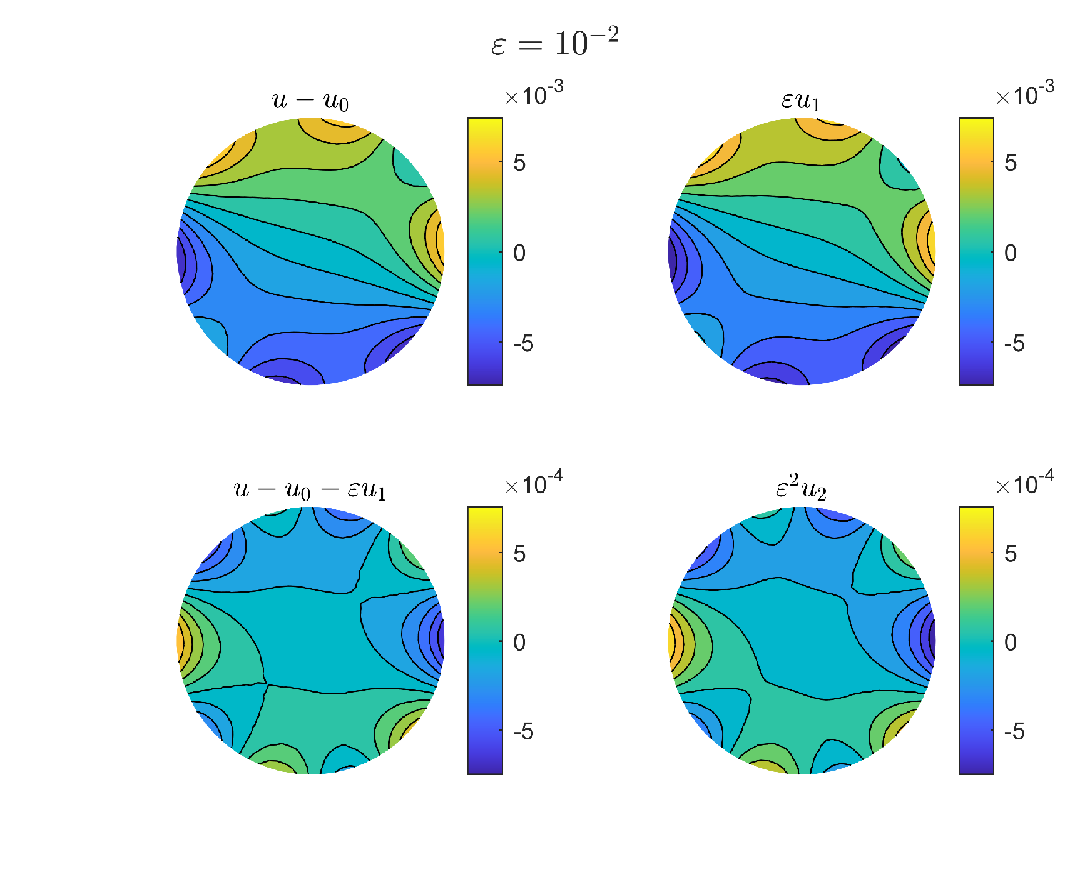}
    \caption{Contour plot of  $u-u_0$ versus $\varepsilon u_1$ and $u - u_0 - \varepsilon u_1$ versus $\varepsilon^2 u_2$ for $\varepsilon = 10^{-2}$ on $(l_c,P) = (1/8,10)$.}
    \label{fig:6_Aubin_RBC_AP_Contour_P_10_eps_1e-2}
\end{figure}

From Figure \ref{fig:6_Aubin_RBC_AP}, we see how the zeroth order error, $\|u-u_0 \|_{L^1}$, converges linearly with $\varepsilon$. Moreover, $\|u-u_0-\varepsilon u_1 \|_{L^1}$ and $\|u-u_0 - \varepsilon u_1 - \varepsilon^2 u_2 \|_{L^1}$, can be observed to converge as expected second and third order with $\varepsilon$, respectively. That is until lower order errors start to dominate, e.g., see for $P = 10$ in the figures. The convergence rate attends spectral behavior until this point. From the contour plots, going through Figure \ref{fig:6_Aubin_RBC_AP_Contour_P_10_eps_1e-2} to \ref{fig:6_Aubin_RBC_AP_Contour_P_10_eps_1e-6}, it is evident how the minimum and maximum values -- in general -- are attained at the boundary of the domain, whereas it is zero in the center. From the figures, $u-u_0$ and $\varepsilon u_1$ look to have more and more sufficient agreement as $\varepsilon$ decreases (minor difference for $\varepsilon = 10^{-2}$ and no visible difference for $\varepsilon = 10^{-4}$ and $\varepsilon = 10^{-6}$) . Moreover, when decreasing $\varepsilon$ the difference between $u - u_0 - \varepsilon u_1$ and $\varepsilon^2 u_2$ becomes more noticeable. At $\varepsilon = 10^{-6}$, $\varepsilon^2 u_2$ is practically machine precision.

\begin{figure}[H]
    \centering
    \includegraphics[scale = 0.6]{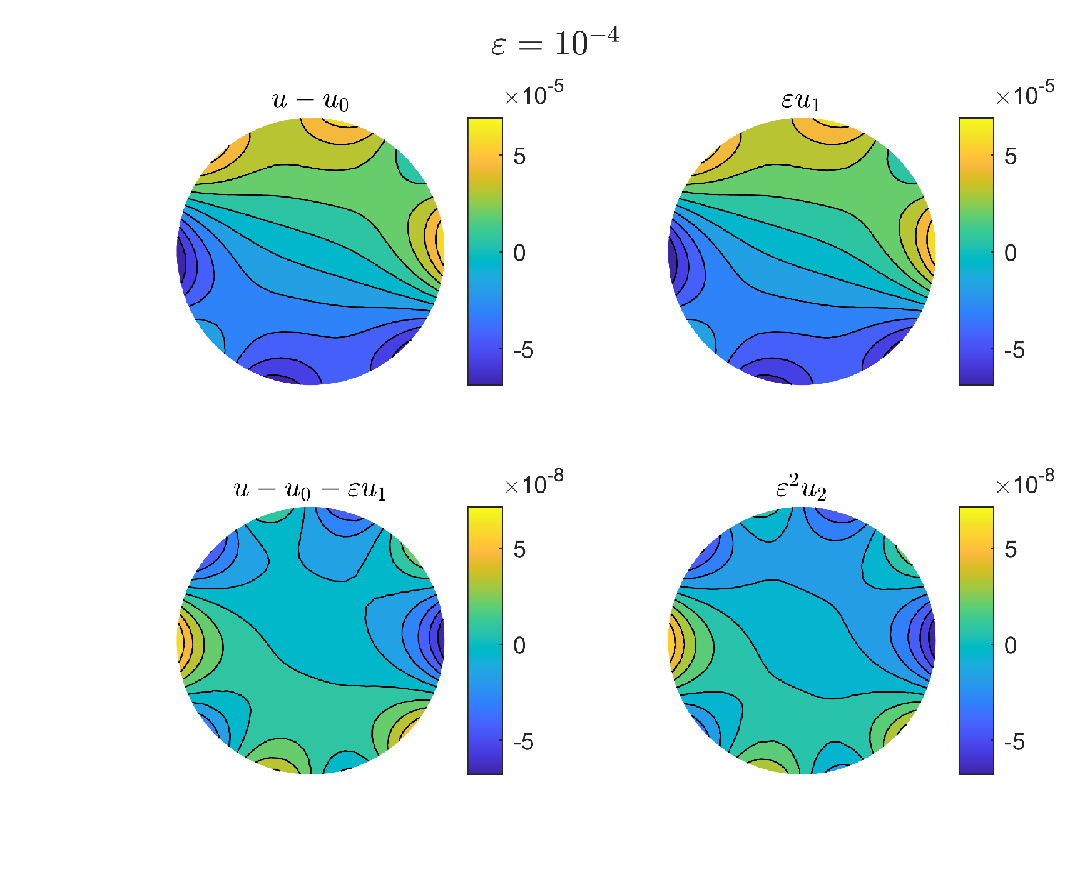}
    \caption{Contour plot of  $u-u_0$ versus $\varepsilon u_1$ and $u - u_0 - \varepsilon u_1$ versus $\varepsilon^2 u_2$ for $\varepsilon = 10^{-4}$ on $(l_c,P) = (1/8,10)$.}
    \label{fig:6_Aubin_RBC_AP_Contour_P_10_eps_1e-4}
\end{figure}

\begin{figure}[H]
    \centering
    \includegraphics[scale = 0.6]{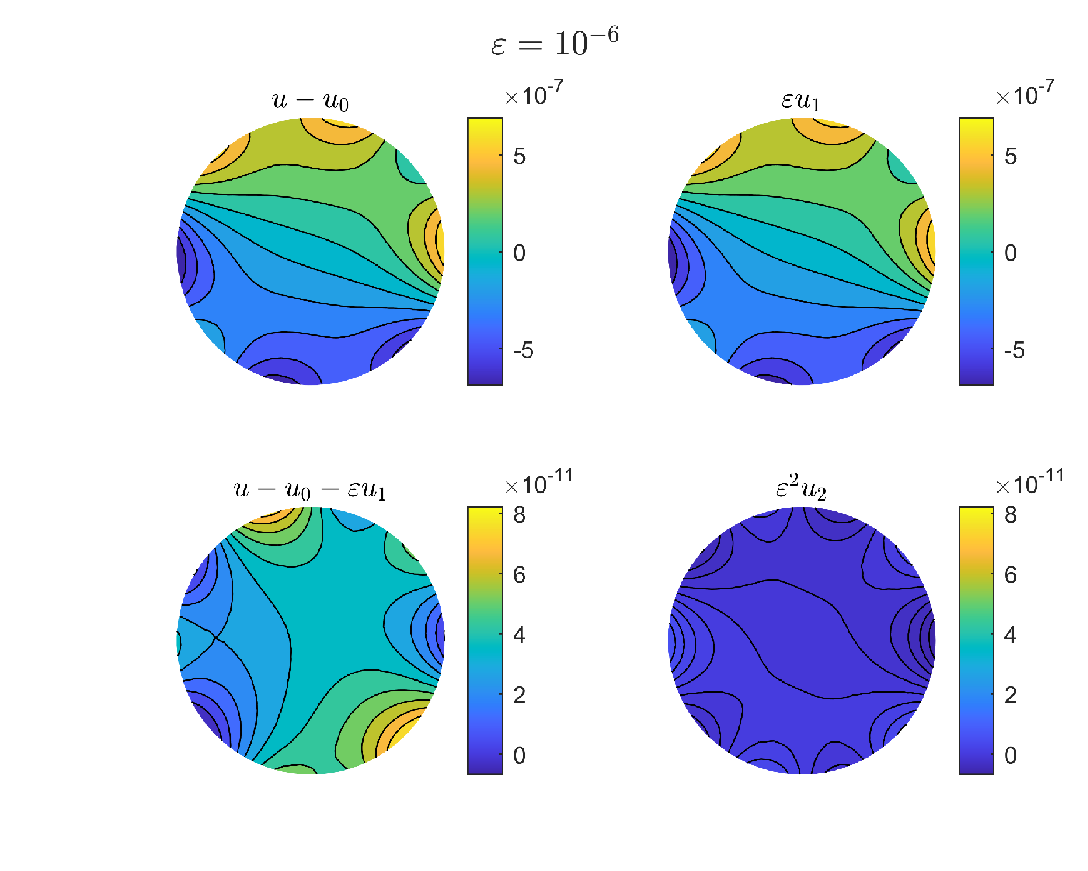}
    \caption{Contour plot of  $u-u_0$ versus $\varepsilon u_1$ and $u - u_0 - \varepsilon u_1$ versus $\varepsilon^2 u_2$ for $\varepsilon = 10^{-6}$ on $(l_c,P) = (1/8,10)$.}
    \label{fig:6_Aubin_RBC_AP_Contour_P_10_eps_1e-6}
\end{figure}

\newpage
 \section{Conclusion}\label{sec:conclusion}

In the paper, we have presented a new high-order spectral element-based numerical model for solving the two-dimensional Poisson problem on an unfitted/immersed/embedded domain. The model utilizes the general idea from the shifted boundary method; however, the cumbersome evaluation of high-order Taylor series expansions is avoided by exploring a polynomial correction. Boundary conditions -- Dirichlet, Neumann, and Robin -- are enforced weakly in both conformal and unfitted settings. In the latter, we have presented a consistent asymptotic preserving Robin formulation for all values of the Robin parameter, $\varepsilon$. Moreover, we have studied the influence of how to construct the surrogate domain. Ultimately showing that interpolating approaches improve convergence rates and decrease conditioning numbers significantly. Various variational formulations -- Aubin and Nitsche type -- have been implemented concluding that as long as the basis functions are being interpolated, the differences between the two are minor.

\section*{Acknowledgments} 
This work contributes to the activities in the PhD research project "New advanced simulation techniques for wave energy converters" hosted at the Department of Applied Mathematics and Computer Science, Technical University of Denmark, Denmark. Work started during a visit for JV at INRIA at the University of Bordeaux. 
 
\appendix

\newpage

\section{Verification on aligned meshes for Dirichlet problems}\label{sec:appendix_DBC}

\subsection{Aubin's method}

\begin{figure}[H]
    \centering
    \includegraphics{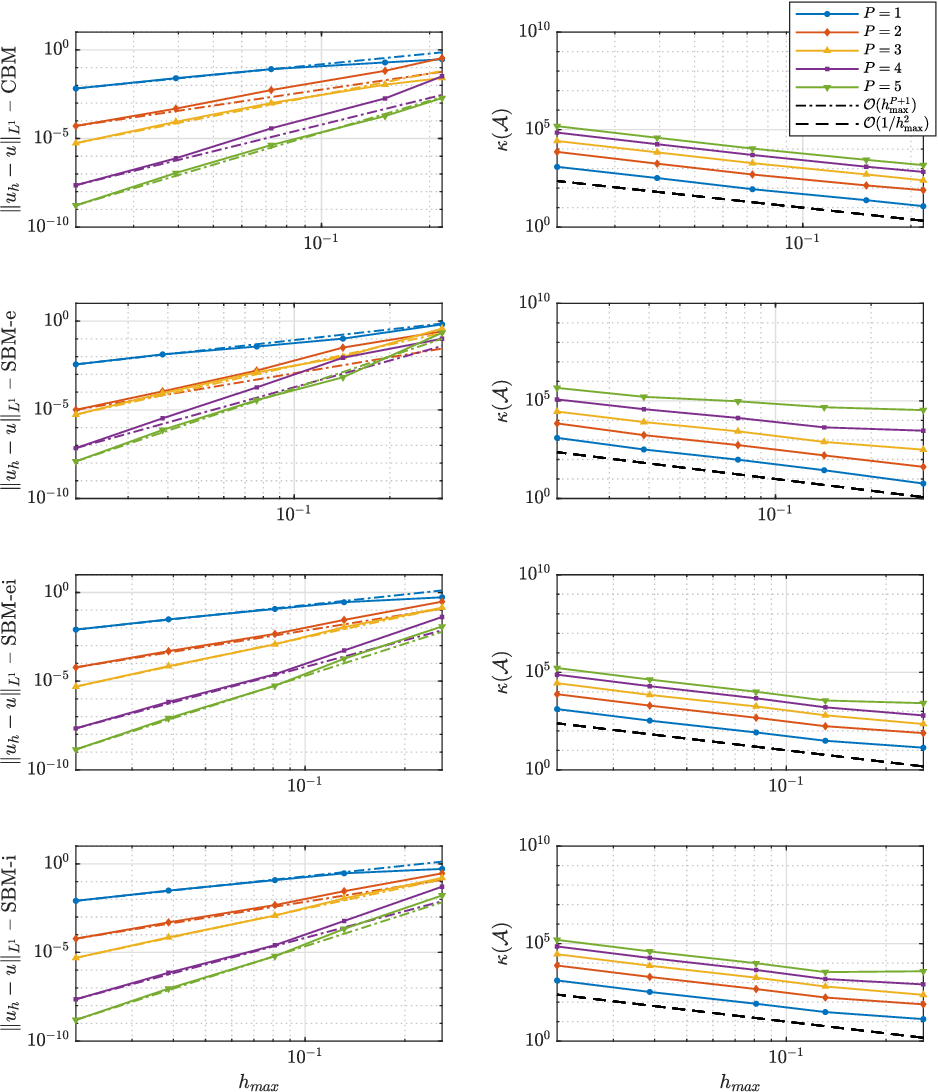}
    \caption{Verification of the Dirichlet Aubin formulation. From the top row and down: CBM, SBM-e, SBM-ei, and SBM-i. Left column is $h$-convergence in $\|u_h - u \|_{L^1}$ and right column is the associated  matrix conditioning, $\kappa(\mathcal{A})$. Analysis is for $P = \{1,...,5\}$ on aligned meshes.}
    \label{fig:1_Aubin_DBC}
\end{figure}

\subsection{Nitsche's method}

\begin{figure}[H]
    \centering
    \includegraphics{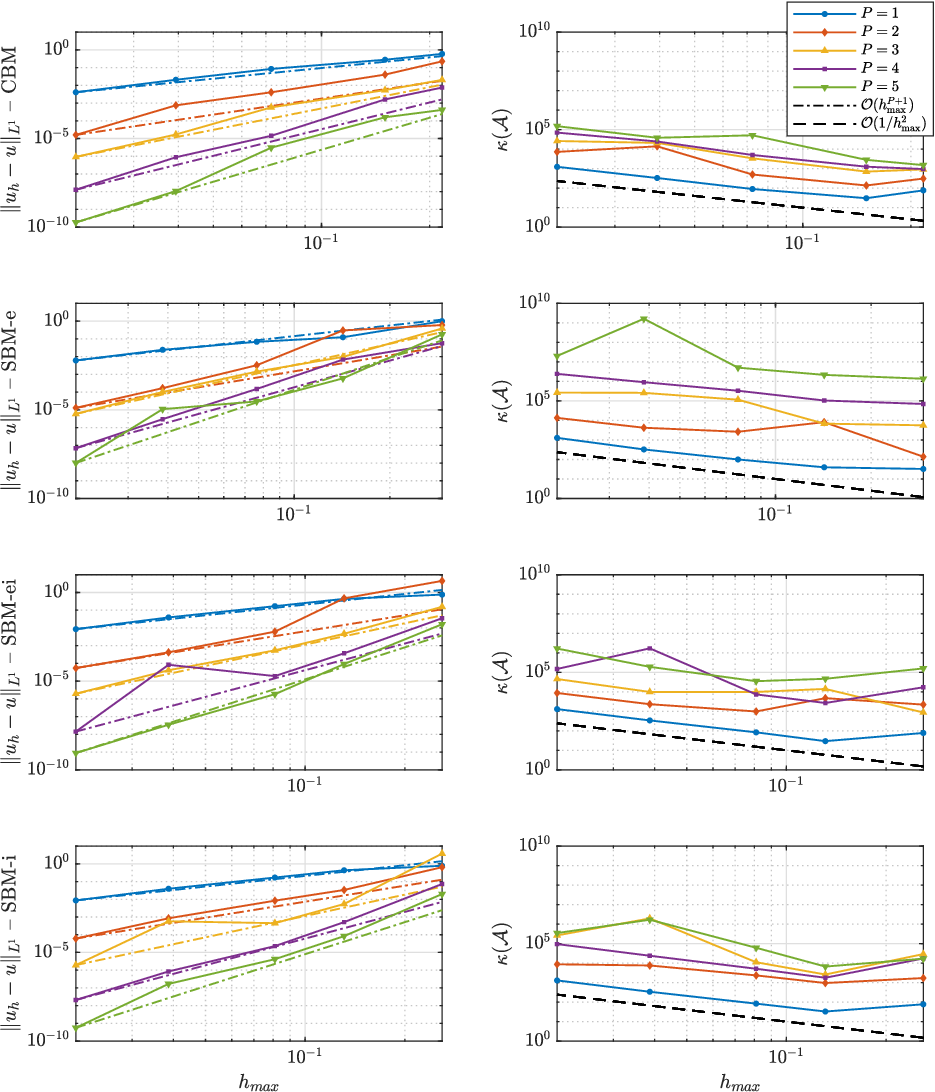}
    \caption{Verification of the Dirichlet Nitsche formulation. From the top row and down: CBM, SBM-e, SBM-ei, and SBM-i. Left column is $h$-convergence in $\|u_h - u \|_{L^1}$ and right column is the associated  matrix conditioning, $\kappa(\mathcal{A})$. Analysis is for $P = \{1,...,5\}$ on well-behaved meshes.}
    \label{fig:1_Nitsche_DBC}
\end{figure}

\section{Assessment on unstructured meshes for Robin boundary conditions}\label{sec:appendix_RBC}

\begin{figure}[H]
    \centering
    \includegraphics{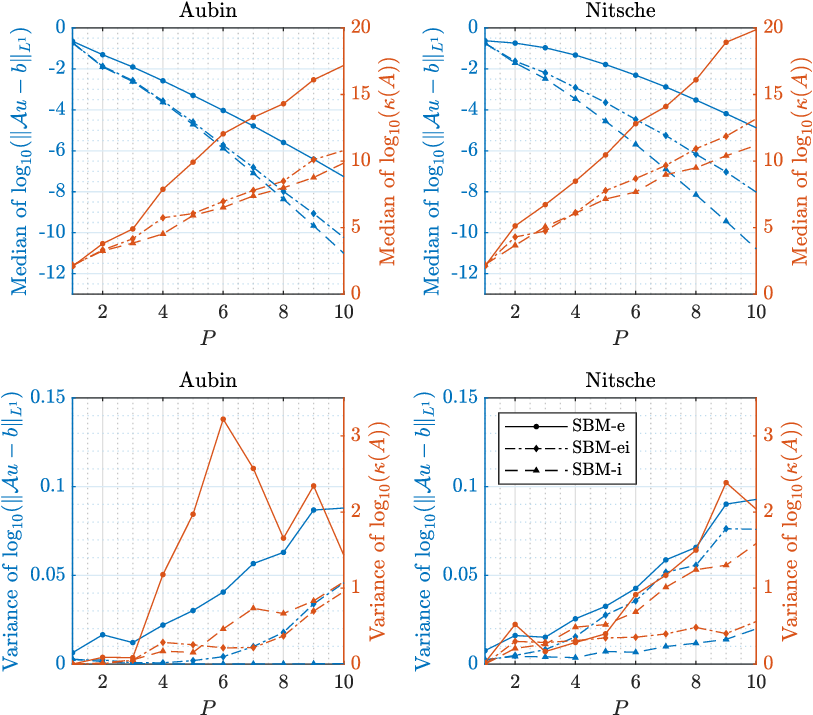}
    \caption{Median (top row) and variance (bottom row) of 100 discrete solutions with the setup as in Figure \ref{fig:Overall_assessment_mesh} for the consistent Robin problem with $\epsilon = 1$. Quantities are for $\|\mathcal{A}u - b \|_{L^1}$ and $\kappa(\mathcal{A})$ for Aubin and Nitsche formulation with SBM-e, SBM-ei, and SBM-i.}
    \label{fig:7_RBC}
\end{figure}

\end{document}